\documentclass[a4paper,11pt, reqno]{amsart}
\usepackage{amsfonts}
\usepackage{latexsym}
\usepackage{amssymb}
\usepackage{amsmath}
\usepackage{color}
\usepackage{bbm}
\usepackage{tikz}
\usepackage{enumerate}
\usepackage{mathrsfs}
\usepackage{todonotes}
\usepackage{mathtools}
\mathtoolsset{showonlyrefs}

\definecolor{darkgreen}{rgb}{0.3,12,0.3}
\definecolor{darkred}{rgb}{1,0.1,0.3}

\usepackage[left=2.5cm, top=2.5cm, bottom=2.5cm, right=2.5cm]{geometry}

\newcommand{\R}{\mathbb R}
\newcommand{\N}{\mathbb N}

\newcommand{\E}{\mathbb E}
\newcommand{\Pro}{\mathbb P}

\newcommand{\Var}{\mathrm{Var}}
\newcommand{\Cov}{\mathrm{Cov}}

\def\dint{\textup{d}}
\newcommand{\SSS}{\ensuremath{{\mathbb S}}}

\newcommand{\B}{\ensuremath{{\mathbb B}}}

\renewcommand{\P}{\mathbb{P}}

\DeclareMathOperator{\spn}{span}

\newtheorem{thm}{Theorem}[section]
\newtheorem*{thm*}{Theorem}
\newtheorem{thmalpha}{Theorem}

\newtheorem{lemma}[thm]{Lemma}

\newtheorem{proposition}[thm]{Proposition}

\newtheorem{rmk}[thm]{Remark}

\usepackage{nomencl}
\makenomenclature

\allowdisplaybreaks

\begin{document}


\title[]{Berry-Esseen bounds for random projections of $\ell_p^n$-balls}

\author[S. Johnston]{Samuel Johnston}
\address{Institut f\"ur Mathematik \& Wissenschaftliches Rechnen, University of Graz, Heinrichstra{\ss}e 36, 8010 Graz, Austria} \email{samuel.johnston@uni-graz.at}

\author[J. Prochno]{Joscha Prochno}
\address{Institut f\"ur Mathematik \& Wissenschaftliches Rechnen, University of Graz, Heinrichstra{\ss}e 36, 8010 Graz, Austria} \email{joscha.prochno@uni-graz.at}




\begin{abstract}
In this work we study the rate of convergence in the central limit theorem for the Euclidean norm of random orthogonal projections of vectors chosen at random from an $\ell_p^n$-ball which has been obtained in [Alonso-Guti\'errez, Prochno, Th\"ale: Gaussian fluctuations for high-dimensional random projections of $\ell_p^n$-balls, Bernoulli 25(4A), 2019, 3139--3174]. More precisely, for any $n\in\N$ let $E_n$ be a random subspace of dimension $k_n\in\{1,\ldots,n\}$, $P_{E_n}$ the orthogonal projection onto $E_n$, and $X_n$ be a random point in the unit ball of $\ell_p^n$. We prove a Berry-Esseen theorem for $\|P_{E_n}X_n\|_2$ under the condition that $k_n\to\infty$. This answers in the affirmative a conjecture of Alonso-Guti\'errez, Prochno, and Th\"ale who obtained a rate of convergence under the additional condition that $k_n/n^{2/3}\to\infty$ as $n\to\infty$. In addition, we study the Gaussian fluctuations and Berry-Esseen bounds in a $3$-fold randomized setting where the dimension of the Grassmannian is also chosen randomly. Comparing deterministic and randomized subspace dimensions leads to a quite interesting observation regarding the central limit behavior. In this work we also discuss the rate of convergence in the central limit theorem of [Kabluchko, Prochno, Th\"ale: High-dimensional limit theorems for random vectors in $\ell_p^n$-balls, Commun. Contemp. Math. (2019)] for general $\ell_q$-norms of non-projected vectors chosen at random in an $\ell_p^n$-ball.
\end{abstract}

\maketitle


\section{Introduction and results}

The last decades have shown the power of probability theory in the study of high-dimensional convex bodies. Although high dimensions typically increase the complexity of problems, it has become apparent that convexity enforces a certain regularity in the geometry of the underlying space and in the behavior of random objects. The area dealing with random phenomena as the space dimension tends to infinity is also known as high-dimensional probability theory and it combines powerful methods from probability theory, functional analysis, and convex geometry. One of the first results in this direction is a high-dimensional central limit theorem commonly known as the Poincar\'e-Maxwell-Borel Lemma (see \cite{DF1987}), showing that the first $k$ coordinates of a point uniformly distributed over the $n$-dimensional Euclidean ball or sphere are independent standard normal variables in the limit as $n\to\infty$ with $k$ fixed. Later Stam studied the central limit behavior when $k$ is allowed to vary with the dimension $n$ \cite{Stam1982}.  Probably the best example of the last decade is a celebrated result of Klartag. He proved a central limit theorem for high-dimensional convex bodies, showing that most $k$-dimensional marginals of a random vector uniformly distributed in an isotropic convex body are approximately Gaussian, provided that $k=k_n$ is smaller than $n^\kappa$ for some absolute constant $\kappa\in (0,1)$ known to satisfy $\kappa< 1/14$ \cite{KlartagCLT, KlartagCLT2}. Other than Klartag's central limit theorem there are only a few random geometric quantities that have been shown to satisfy a central limit theorem. For instance, central limit phenomena for the $\log$-volume of random simplices in high dimensions have recently been studied by Grote, Kabluchko, and Th\"ale in \cite{GKT2019}. Then there is the central limit theorem for the volume of $k$-dimensional random projections of the $n$-dimensional cube obtained by Paouris, Pivovarov, and Zinn in \cite{PPZ14}, which in the case $k=1$ was obtained independently by Kabluchko, Litvak, and Zaporozhets \cite{KLZ2015}. In fact, when $k=1$ we can reformulate the result into a central limit theorem for the $\ell_1^n$-norm $\Vert\theta\Vert_1$, where $\theta$ is a random vector uniformly distributed on the Euclidean sphere $\SSS^{n-1}$. This has recently been extended by Kabluchko, Prochno, and Th\"ale to a central limit theorem for arbitrary $\ell_p^n$-balls, $1\leq p\leq\infty$ \cite{KPT2019, KPT2019_II}. Gaussian fluctuation in the Grassmannian setting were then studied by Alonso-Guti\'errez, Prochno, and Th\"ale in \cite{APT2019}. They proved a central limit theorem for the Euclidean norm of random orthogonal projections and also obtained Berry-Esseen bounds on the rate of convergence in their central limit theorem whenever the subspace dimensions grow fast enough as the dimension $n$ of the ambient space tends to infinity.

In this paper we investigate the rate of convergence in the Grassmannian setting  studied in \cite{APT2019}. For $n \in \mathbb{N}$, denote by $E_n$ a random $k_n$-dimensional subspace of $\mathbb{R}^n$, by $P_{E_n}$ the orthogonal projection onto $E_n$, and by $X_n$ a random point in the unit ball of $\ell_p^n$. Alonso-Guti\'errez, Prochno, and Th\"ale proved a Berry-Esseen theorem for the sequence $\|P_{E_n}X_n\|_2$, $n\in\N$ of random variables under the condition that the subspace dimensions $k_n$ tend to infinity and satisfy $k_n/n^{2/3}\to\infty$ as $n\to\infty$. Here we weaken the assumption considerably by showing that the condition $k_n/n^{2/3}\to\infty$ as $n\to\infty$ may be lifted. Contrary to the more analytic approach pursued in \cite{APT2019}, we use a probabilistic one based on the geometry of Kolmogorov-Smirnov distances and the relationship between different Gaussian densities under this metric. One of the ingredients in our approach is giving a bound on the Kolmogorov-Smirnov distance between different Gaussians in terms of the variances. This can be seen as a version of a result appearing in a paper of Brehm and Voigt \cite{BV}. This tool teams up with basic properties of Kolmogorov-Smirnov distances and the classical Berry-Esseen theorem in allowing us to establish convergence in the Grassmannian setting studied in \cite{APT2019}. 

We also discuss a related problem to the one studied in \cite{APT2019}, where the dimension $K_n$ of the Grassmannian is chosen randomly according to a binomial distribution. Here an underlying probabilistic structure makes certain aspects of this $3$-fold randomized problem simpler than its ``deterministic'' counterpart: the sum of $K_n$ random variables may be represented as a sum of $n$ random variables multiplied by indicator functions, and as a result the key random variables involved lend themselves more easily to use in the Berry-Esseen theorem.

In addition to the Grassmannian setting, we provide a Berry-Esseen type central limit theorem for the $\ell_q^n$-norm of a random vector in an $\ell_p^n$ ball, where $0< p \leq \infty $ and $q\neq p$ with $0<q <\infty$. This setting has been investigated by Kabluchko, Prochno, and Th\"ale in \cite{KPT2019, KPT2019_II}. For $q=2$ the Berry-Esseen bound already follows from  \cite[Theorem 1.2]{APT2019}, when the subspace dimensions $k_n$ are chosen to be $n$ (and so $k_n/n^{2/3}$ tends to infinity as $n$ does).

\subsection{The distributions on the $\ell_p^n$-balls}
Before stating our main results, we need to introduce our class of probability measures on unit balls of $\ell_p^n$, following the description in \cite{APT2019, BartheGuedonEtAl}. Let $||x||_p := \left( \sum_{ i = 1}^n |x_i|^p \right)^{1/p}$ be the $p$-norm, and let  $\mathbb{B}_p^n$ be the unit ball of $\ell_p^n$. Let $\mathbb{U}_{n,p}$ be the uniform distribution on $\mathbb{B}_p^n$ and let $\mathbb{C}_{n,p}$ be the cone probability measure on its boundary. (We give a full definition for the cone measure on Section \ref{sec:geom}.) Let $\mathbb{W}$ be any Borel probability measure on $[0,\infty)$. In this paper we consider distributions on $\mathbb{B}_p^n$ of the form
\begin{align} \label{P def}
\mathbb{P}_{n,p,\mathbb{W}} := \mathbb{W} \left( \{ 0 \} \right) \mathbb{C}_{n,p} + H \mathbb{U}_{n,p},
\end{align}
where $H: \mathbb{B}_p^n\to\R$ is given by $H(x) = h(||x||_p)$ with
\begin{align} \label{h def}
h(r) := \frac{1}{ \Gamma (1 + \frac{n}{p} ) } \frac{1}{ \left(1 - r^p \right)^{1 + n/p} } \int_0^\infty s^{n/p} e^{ - r^p s / (1 - r^p) } \mathbb{W}(\dint s), \qquad r \in [0,1]. 
\end{align}
In other words this means that
\begin{align*}
\int_{\B_p^n}f(x)\,\mathbb{P}_{n,p,\mathbb{W}}(\dint x) &= \mathbb{W} \left( \{ 0 \} \right)\int_{\SSS_p^{n-1}}f(x)\,\mathbb{C}_{n,p}(\dint x) + \int_{\B_p^n}f(x)\,H(x)\,\mathbb{U}_{n,p}(\dint x)\\
\end{align*}
for all non-negative measurable functions $f:\B_p^n\to\R$.

It is a plain that when $\mathbb{W}$ is the Dirac mass at $0$ the measure $\mathbb{P}_{n,p,\mathbb{W}}$ coincides with the cone probability measure $\mathbb{C}_{n,p}$. On the other hand, the choice
\begin{align*}
\mathbb{W}( \dint s) := p^{-1}e^{ - s/p}  \dint s 
\end{align*}
yields that $H\equiv 1$ and $\mathbb{W} \left( \{ 0 \} \right)=0$, and hence in this case $\mathbb{P}_{n,p,\mathbb{W}}$ is equal to the uniform measure on $\mathbb{B}_p^n$. Let us also mention here that when $\mathbb{W}$ is the gamma distribution with a certain parameter, a random vector with law $\mathbb{P}_{n,p,\mathbb{W}}$ may be understood as an $n$ dimensional projection of a random element of $\mathbb{B}_p^{n+m}$ \cite{APT2019, BartheGuedonEtAl}. With these special cases in mind, for the remainder of this article we will work in full generality, with the probability measure $\mathbb{W}$ at hand.

Define
\begin{align*}
M_p(r) := \frac{ p^{r/p} }{ r+1} \frac{ \Gamma(1+\frac{r+1}{p} )}{ \Gamma(1+\frac{1}{p} ) }.
\end{align*} 
As we will see below (Lemma \ref{p Gaussian moments}), the constants $M_p(r)$ correspond to the moments of the so-called $p$-Gaussian distribution, and will  feature in all of our results. Furthermore, in all three of our main results, Theorem \ref{thm:RP 1}, Theorem \ref{thm:RP 2}, and Theorem \ref{thm:q norm} we will encounter the quantity
\begin{align} \label{variance}
\sigma^2(p,q) := \frac{1}{q^2} \left( \frac{ M_p(2q)}{ M_p(q)^2} - 1 \right) - \frac{2}{pq} \left( \frac{M_p(p+q)}{M_p(q)} - 1 \right) + \frac{1}{p^2} \left( M_p(2p) - 1 \right). 
\end{align}
which is the variance of the centered random variable $\frac{ |Z|^q - M_p(q)}{ q M_p(q)} - \frac{ |Z|^p - 1}{ p }$ where $Z$ is $p$-Gaussian. We emphasize that $\sigma^2(p,q)$ is not symmetric in $p$ and $q$.

Finally, the Kolmogorov-Smirnov distance between a pair of random variables is given by the $\sup$-norm distance between their respective distribution functions. Namely, if $X$ and $Y$ are real-valued random variables, the Kolmogorov-Smirnov distance between $X$ and $Y$ is simply
\begin{align} \label{KS def}
\dint_{\mathsf{KS}}(X,Y) := \sup_{ t \in \mathbb{R}}  \big| \P( X \leq t )  - \P( Y \leq t ) \big| .
\end{align}
With these definitions at hand, we are now equipped to state our main results. 

\subsection{The size of random projections of random elements of $\mathbb{B}_p^n$}
We denote by $\mathbb{G}_{n,k}$ the Grassmannian manifold of $k$-dimensional subspaces of $\mathbb{R}^n$. $\mathbb{G}_{n,k}$ can be identified with $\mathcal O(n)/(\mathcal O(k)\times \mathcal O(n-k))$, where $\mathcal O(m)$ stands for the orthogonal group in dimension $m$. There exists a unique probability measure $\nu_{n,k}$ on $\mathbb{G}_{n,k}$ that is invariant under the action of the orthogonal group. The measure $\nu_{n,k}$ has the following construction. Let $\{e_i\}_{1 \leq i \leq n }$ be the standard unit vector basis for $\mathbb{R}^n$, and let $E_k$ denote the span of $\{e_i\}_{1 \leq i \leq k}$. Now let $T$ be a Haar distributed element of $\mathcal O(n)$. Then the (random) image $T E_k$ of the subspace $E_k$ under $T$ has distribution $\nu_{n,k}$. 

We now recall the main result of \cite{APT2019}, concerning the large-$n$ asymptotics of the Euclidean sizes of projections of random variables with law $\mathbb{P}_{n,p,\mathbb{W}}$ onto random subspaces with law $\nu_{n,k}$. 
Let $p \neq 2$ and $k_n$ be a sequence of integers such that $1 \leq k_n \leq n$ and $k_n\to\infty$. Assume  that $\lambda_n := \frac{ k_n }{ n }$ converges to some $\lambda \in [0,1]$ as $n \to \infty$. Let $X_n$ be a random vector with law $\mathbb{P}_{n,p,\mathbb{W}}$, and let $E_n$ be a random subspace of $\mathbb{R}^n$ with law $\nu_{n,k_n}$, independent of $X_n$. Define the random variable
\begin{align}\label{def: Y_n}
Y_n := \frac{ n^{1/p}  }{ \sqrt{M_p(2)}} || P_{E_n} X_n||_2 - \sqrt{k_n}.
\end{align}
Then \cite[Theorem 1.1]{APT2019} states that $Y_n$ converges in distribution to a Gaussian random variable $G$  with variance $\lambda \sigma^2(p,2) + \frac{1}{2}(1-\lambda)$,
where $\sigma^2(p,2)$ is obtained by setting $q=2$ in \eqref{variance}. 
The authors of \cite{APT2019} also provide a Berry-Esseen type result. Namely, \cite[Theorem 1.2]{APT2019} states that there are constants $c_p,C_p\in(0,\infty)$ depending only on $p$ such that 
\begin{align} \label{APT bound}
\dint_{\mathsf{KS}}(Y_n,G ) \leq C_p  \max \left\{ \frac{ \log (k_n)}{ \sqrt{k_n} } ,  \frac{n}{ k_n^{3/2}}, | \lambda_n - \lambda| ,  \P \Big[  W  > c_p \frac{ n \log (k_n)}{ k_n} \,\Big] \right\},
\end{align} 
where $W$ is a non-negative random variable distributed according to $\mathbb{W}$.
We remark that the efficacy of the bound \eqref{APT bound} breaks down in the case where $\lambda = 0$ and $k_n$ grows to infinity slower than $n^{2/3}$. The first main result of the present article is a considerable strengthening of \eqref{APT bound} and states that the condition $k_n/n^{2/3}\to\infty$ can be removed.
\begin{thmalpha} \label{thm:RP 1}
Let $1\leq p <\infty$ and $p\neq 2$. Let $Y_n$ be a random variable defined as in \eqref{def: Y_n} where $k_n\to\infty$ and $\lambda_n := k_n/n$ converges to $\lambda\in[0,1]$. Then 
\begin{align*}
\dint_{\mathsf{KS}}(Y_n,G ) \leq C_p  \max \left\{ \frac{ \log (k_n)}{ \sqrt{k_n} } ,\,  | \lambda_n - \lambda| ,\,  \P \left[  W  > c_p \frac{ n \log (k_n)}{ k_n}  \right] \right\}\,,
\end{align*} 
where $W$ is a non-negative, independent random variable with distribution $\mathbb{W}$ and $c_p, C_p \in(0,\infty)$ are constants depending only on $p$.
\end{thmalpha}

\subsection{Random projections onto subspaces of random dimension} \label{sec:RP 2 statement}

We will also consider subspace-valued random variables of random dimension. For $n\in\N$, $k\in\{0,\dots,n \}$, and $\lambda\in[0,1]$ consider the probability measure $\mu_{n,\lambda}$ on the set $\mathbb{G}_n := \cup_{k=0}^n \mathbb{G}_{n,k}$ of all subspaces of $\mathbb{R}^n$ defined as
\begin{align*}
\mu_{n,\lambda} ( \mathcal{E}_{n,k} ) = \binom{n}{k} \lambda^k (1- \lambda)^{n-k} \nu_{n,k}( \mathcal{E}_{n,k})\,,
\end{align*}
whenever $\mathcal{E}_{n,k}$ is a measurable subset of $\mathbb{G}_{n,k}$. In particular, if $E_n$ is distributed according to $\mu_{n,\lambda}$, then the dimension of $E_n$ is binomially distributed with parameters $n$ and $\lambda$. 

The measure $\mu_{n,\lambda}$ may be constructed as follows. Let $\{e_i\}_{1 \leq i \leq n}$ be the standard unit vector basis of $\mathbb{R}^n$, and let $I_\lambda$ be the random subset of $\{1,\ldots,n\}$ generated by independently including each element $1 \leq i \leq n $ in $I_\lambda$ with probability $\lambda$, and not including $i$ with probability $1-\lambda$. Now consider the random subspace 
\begin{align*}
E_\lambda := \spn \big\{ e_i : i \in I_\lambda \big\}
\end{align*}
of $\mathbb{R}^n$ with random dimension ${\#}I_\lambda$, where $\#$ denotes the cardinality of a finite set. Finally, let $T$ be a Haar distributed element of the orthogonal group of dimension $n$. Then the image $T E_\lambda$ of $E_\lambda$ under $T$ is distributed according to $\mu_{n,\lambda}$. 

The following theorem concerns the asymptotic Euclidean size of a random vector with law $\mathbb{P}_{n,p,\mathbb{W}}$ projected onto a random subspace $E_n$ of distribution $\mu_{n, \lambda_n}$. The sequence $(\lambda_n\, n)_{n\in\N}$ appearing in the theorem may be interpreted as the average dimension of a subspace with distribution $\mu_{n ,\lambda}$ we project onto.

\begin{thmalpha} \label{thm:RP 2}
Let $1\leq p < \infty$ and $p \neq 2$. Let $\lambda_n$ be a sequence of elements of $(0,1]$ such that $\lambda_n$ converges to some $\lambda \in [0,1]$. Let $X_n$ be a random variable with law $\mathbb{P}_{n,p,\mathbb{W}}$, and let $ E_n$ be a random subspace of $\mathbb{R}^n$ with law $\mu_{n,\lambda_n}$ which is independent of $X_n$. Define the random variable
\begin{align*}
Y_n := \frac{ n^{1/p}  }{ \sqrt{M_p(2)}} || P_{E_n} X_n||_2 - \sqrt{\lambda_n n}.
\end{align*}
Then, as $n \to \infty$, the random variables $Y_n$ converge in distribution to a centered Gaussian random variable $G$ with variance $\lambda \sigma^2(p,2) + \frac{3}{4}( 1- \lambda)$. 
Moreover,  
\begin{align*}
\dint_{\mathsf{KS}}( {Y}_n  , G) \leq C_p  \max \left\{ \frac{ \log(n) }{ \sqrt{ \lambda_n n } } , | \lambda_n - \lambda| , \P \left[ W >  c_p \frac{1}{ \lambda_n } \log(n) \right] \right\}\,,
\end{align*}
where $W$ is a non-negative, independent random variable with distribution $\mathbb{W}$ and $c_p,C_p\in(0,\infty)$ are constants depending only on $p$.
\end{thmalpha}

We conclude this subsection with two interesting observations that arise from a comparison of Theorem~\ref{thm:RP 1} and Theorem~\ref{thm:RP 2}.

\begin{rmk}
Comparing the variances that occur in Theorem~\ref{thm:RP 1} and Theorem~\ref{thm:RP 2}, we notice that the one in the latter is larger. This however seems natural given the increased randomness in that setting. 

Moreover, we find it particularly interesting and surprising that the additional randomness gives rise to a term $-\frac{3}{4}\lambda$ in the variance, because exactly such an additional term also arises in the comparison of a central limit theorem for random projections and deterministic projections in \cite{KPT2019_II} (see discussion after Theorem E in that paper).  
\end{rmk}

\begin{rmk}
Theorem~\ref{thm:RP 1} together with Theorem~\ref{thm:RP 2} leads us to the observation that we have the same central limit theorem regardless of whether we project onto uniform random subspaces of dimensions $k_n$ or onto random subspaces of random dimension, provided their dimensions are sufficiently large, i.e., $\lambda_n\to 1$ and $k_n/n\to 1$ as $n\to\infty$. 
On the other hand, if $\lambda_n$ coincides with the sequence $k_n/n$ in Theorem \ref{thm:RP 1} and $\lambda_n\to\lambda\in [0,1)$, then we still have central limit theorems with the same centering, but this time with different limiting variances.
\end{rmk}

\subsection{The $q$-norm of random elements of $\mathbb{B}_p^n$}

Finally, we study the asymptotics of the $\ell_q$-norm of a random element of $\mathbb{B}_p^n$ as the dimension $n$ tends to infinity. To this end, let us first recall an abbreviated version of the main result in \cite[Theorem 1.1 (a), $d=1$]{KPT2019}. 
Assume that $1 \leq p \neq q < \infty$ and let $X_n$ be a random vector uniformly distributed on $\mathbb{B}_p^n$, and define the random variable
\begin{align} \label{KPT Y}
Y_n := \sqrt{n} \left( \frac{n^{1/q - 1/p} }{ M_p(q)^{1/q}   } ||X_n||_q - 1 \right).
\end{align}
Then, as $n \to \infty$, $Y_n$ converges in distribution to $G$, a centered Gaussian random variable with variance $\sigma^2(p,q)$ as given in \eqref{variance}.
(The statement we give here is abbreviated in the sense that the authors of \cite{KPT2019} work in a multivariate setting, considering the simultaneous convergence for different values of $q$.)
Our final result, Theorem \ref{thm:q norm}, establishes an upper bound on the rate at which $Y_n$ converges to $G$  in terms of the Kolmogorov-Smirnov distance. In fact, we opt to work in a slightly more general setting, letting the random vector $X_n$ have distribution $\mathbb{P}_{n,p,\mathbb{W}}$ for a general choice of $\mathbb{W}$ as was considered recently in \cite{KPT2019_II}. 

\begin{thmalpha} \label{thm:q norm}
Let $1 \leq p \neq q < \infty$. Let $X_n$ be a random vector with distribution $\mathbb{P}_{n,p,\mathbb{W}}$, $Y_n$ be defined as in \eqref{KPT Y}, and let $G$ be a centered Gaussian random variable with variance $\sigma^2(p,q)$. Then,
\begin{align} \label{1 BE}
\dint_{\mathsf{KS}}(Y_n , G) \leq \frac{ C_{p,q} \log n}{ \sqrt{n}} +  \P \left[ W > c_{p,q} \sqrt{ n \log n } \right],
\end{align}
where $W$ is a non-negative, independent random variable with distribution $\mathbb{W}$ and $c_{p,q},C_{p,q}\in(0,\infty)$ depend only on $p$ and $q$.
In particular, the sequence $Y_n$ converges in distribution to $G$.
\end{thmalpha}

Since the choice $\mathbb{W}(\dint s) = p^{-1}e^{-s/p}\,\dint s$ results in $\mathbb{P}_{n,p,\mathbb{W}}$ being the uniform distribution on the unit ball $\mathbb{B}_p^n$, Theorem \ref{thm:q norm} implies the central limit theorem of Kabluchko, Prochno, and Th\"ale in \cite{KPT2019} and in \cite{KPT2019_II} for $W_n\equiv W$ there. Here we refrain from working in this slightly more general setting where $W$ may depend on $n$.

\subsection{Overview}

The remainder of the paper is structured as follows. In Section \ref{sec:PrelimNotation}, we introduce our notation, as well as some preliminary facts about the geometry of $\ell_p^n$-balls -- including several probabilistic representations for random vectors that will be integral to our approach. We also initiate here our study of Kolmogorov-Smirnov distances, developing some basic facts like a triangle inequality as well as a Lipschitz continuity result for the variances of different Gaussians. These general facts will be used in the remainder of the article in Section \ref{sec:RP 1 proof}, Section \ref{sec:RP 2 proof} and Section \ref{sec:q norm proof}, which are dedicated to proving Theorem \ref{thm:RP 1}, Theorem \ref{thm:RP 2}, and Theorem \ref{thm:q norm} respectively.

\section{Preliminaries and notation}\label{sec:PrelimNotation}

In this section we shall briefly present the necessary background material together with the notation we use throughout the text.

\subsection{Notation}
In this article we work with $\mathbb{R}^n$ equipped with the standard Euclidean structure. Given a subset $A$ of $\mathbb{R}^n$ and a subset $I$ of the real line, we define
\begin{align} \label{mink mult}
IA := \big\{ r a \in\R^n \,:\, r \in I, a \in A \big\}.
\end{align}
For $k\in\{1,\dots,n\}$, we will denote by $\mathbb{G}_{n,k}$ the Grassmannian of $k$-dimensional subspaces of $\mathbb{R}^n$. For $k=0$, $\mathbb{G}_{n,0}$ is simply the set containing the trivial vector space $\{0\}$, where $0$ is the origin. It is possible to endow $\mathbb{G}_{n,k}$ with a unique rotationally invariant probability measure $\nu_{n,k}$, the Haar probability measure on $\mathbb{G}_{n,k}$. For a subspace $E$ of $\mathbb{R}^n$, we denote by $P_E$ the orthogonal projection of $\mathbb{R}^n$ onto $E$ (with respect to the standard Euclidean structure).

For $1\leq p<\infty$, we define the $\ell_p$-norm of an element $x = (x_1,\ldots,x_n)$ of $\mathbb{R}^n$ by
\begin{align*}
||x||_p := \left( \sum_{i=1}^n |x_i|^p \right)^{1/p}\,.
\end{align*}
For $p \geq 1$, whenever $\mathbb{R}^n$ is endowed with the norm $|| \cdot ||_p$ it becomes a Banach space which we shall denote $\ell_p^n$. The unit ball of $\ell_p^n$ is given by $\mathbb{B}_p^n := \{ x \in \mathbb{R}^n : ||x||_p \leq 1 \}$. We shall write $\SSS_p^{n-1}:= \{ x \in \mathbb{R}^n : ||x||_p = 1 \}$ for the unit sphere in $\ell_p^n$. 

Given a probability space $(\Omega,\mathcal A, \P)$ and a probability measure $\mu$ on a measurable space $(E,\mathscr E)$, we shall indicate by $X\sim \mu$ that the random object $X:\Omega\to E$ has distribution $\mu$. Given any pair of random variables $X$ and $Y$, we denote by
\[X\stackrel{\dint}{=} Y\] 
their equality in distribution.

Finally, we adopt the following convention for constants. We will write $c_p, C_p$ and $c_{p,q}, C_{p,q}$ to denote positive constants depending only on $p$ and on $p,q$ respectively. Within subsections we will distinguish between constants by writing $C_p, C'_p, C_p''$ etc, but two instances of $C_p$ or $C_p'$ occurring in different sections will denote different constants. Constants $c,C\in(0,\infty)$ etc.~without any subindex are absolute constants.

\subsection{The $p$-Gaussian distribution} \label{sec:geom}
We say a real-valued random variable has the $p$-Gaussian distribution if its density with respect to the Lebesgue measure on $\R$ is given by
\begin{align*}
\phi_p(s) := \frac{ e^{ - |s|^p/p} }{ 2p^{1/p} \Gamma(1 + 1/p) }\dint s
\end{align*}
In the case $p=2$, we say that the random variable has the standard Gaussian distribution.

The following Lemma, lifted from \cite{APT2019} (see also \cite[Lemma 4.1]{KPT2019}), gives the absolute moments of $p$-Gaussian random variables in terms of the gamma function.
\begin{lemma}[Lemma 3.1 \cite{APT2019}] \label{p Gaussian moments}
Let $1\leq p \leq \infty$ and $r\geq 0$. Let $Z$ be $p$-Gaussian distributed. Then
\begin{align*}
 \E\big[ |Z|^r \big] = M_p(r) := \frac{ p^{r/p} }{ r+1} \frac{ \Gamma(1+\frac{r+1}{p} )}{ \Gamma(1+\frac{1}{p} ) }
\end{align*}
and for $q\geq 0$,
\[
\Cov\big[|Z|^q,|Z|^r\big] = M_p(q+r) - M_p(q)M_p(r)\,.
\]
\end{lemma}
In light of Lemma \ref{p Gaussian moments}, it is straightforward to prove that with $\sigma^2(p,q)$ as in \eqref{variance}, if $Z$ is $p$-Gaussian distributed, then
\begin{align*}
\Var\left[ \frac{ |Z|^q - M_p(q)}{ q M_p(q)} - \frac{ |Z|^p - 1}{ p }  \right] = \mathbb{E} \left[ \left( \frac{ |Z|^q - M_p(q)}{ q M_p(q)} - \frac{ |Z|^p - 1}{ p } \right)^2 \right] = \sigma^2(p,q).
\end{align*}

\subsection{Probabilistic representations for measures on $\mathbb{B}_p^n$}
We now consider probability measures on the unit ball $\mathbb{B}_p^n$ and its boundary $\mathbb{S}_p^{n-1}$. Let $\mathbb{U}_{n,p}$ be the uniform measure on $\mathbb{B}_p^n$. We define  the cone measure $\mathbb{C}_{n,p}$ on the boundary $\mathbb{S}_p^{n-1}$ to be the unique probability measure with the property that
\begin{align*}
\mathbb{C}_{n,p}(A) := \mathbb{U}_{n,p} \left( [0,1] A \right),
\end{align*}
for measurable subsets $A$ of $\mathbb{S}_p^{n-1}$, where $[0,1]A$ is defined as in \eqref{mink mult}. Recall that we defined the class of probability measures $\mathbb{P}_{n,p,\mathbb{W}}$ on the unit ball given by
\begin{align*}
\mathbb{P}_{n,p,\mathbb{W}} := \mathbb{W} \left( \{ 0 \} \right) \mathbb{C}_{n,p} + H \mathbb{U}_{n,p},
\end{align*}
where $H$ is a $p$-radial function given by \eqref{h def}.

The first key result of this section is a lemma from \cite{BartheGuedonEtAl}, stating that a random vector distributed according to $\mathbb{P}_{n,p,\mathbb{W}}$ may be represented in terms of independent $p$-Gaussian random variables. 

\begin{lemma}[\cite{BartheGuedonEtAl}, Theorem 3] \label{measure rep}
Let $Z = (Z_1,\ldots,Z_n)$ be a random vector whose entries are independent and identically distributed $p$-Gaussians, and let $W$ be a non-negative real-valued random variable with law $\mathbb{W}$. Then the $n$-dimensional random vector
\begin{align*}
\frac{Z}{ \left( ||Z||_p^p + W \right)^{1/p} } 
\end{align*}
is distributed according to $\mathbb{P}_{n,p,\mathbb{W}}$
\end{lemma}
We remark that in particular, Lemma \ref{measure rep} clarifies that the measure $\mathbb{P}_{n,p,\mathbb{W}}$ is in fact a probability measure---a property that is not immediately obvious from the definition. 

Alonso-Guti\'errez, Prochno, and Th\"ale \cite{APT2018} used Lemma \ref{measure rep} in conjunction with the rotational invariance of standard Gaussian random variables to give the following probabilistic representations for the Euclidean norm of orthogonal projections of random vectors onto random subspaces. Proposition \ref{projection rep} first appeared as \cite[Theorem 3.1]{APT2018}, though the precise formulation we use is \cite[Proposition 2.7]{APT2019}.

\begin{proposition} \label{projection rep}
Let $X_n$ be a random vector in $\mathbb{R}^n$ with law $\mathbb{P}_{n,p,\mathbb{W}}$, and let $E_n$ in $\mathbb{G}_{n,k}$ be a random subspace distributed according to $\nu_{n,k}$, independent of $X_n$. Then
we have the distributional equality
\begin{align} \label{fixed identity}
|| P_{E_n} X_n ||_2 \stackrel{\dint}{=} \frac{ \Big( \sum_{ i =1}^n Z_i^2 \Big)^{1/2} \Big( \sum_{ i =1}^k g_i^2 \Big)^{1/2} }{ \Big( \sum_{ i =1}^n |Z_i|^p + W \Big)^{1/p} \Big( \sum_{ i =1}^n g_i^2 \Big)^{1/2}  },
\end{align}
where $Z_i$ are a collection of independent and identically distributed $p$-Gaussian random variables, and $g_i$ are a collection of standard Gaussian random variables. 
\end{proposition}

Recall that in Section \ref{sec:RP 2 statement} we introduced the probability measure $\mu_{n,\lambda}$ defined on the set of all subspaces of $\mathbb{R}^n$.  We require an analogue of Proposition \ref{projection rep} where the random subspace is chosen according to the probability measure $\mu_{n,\lambda}$ rather than the Grassmannian measure $\nu_{n,k}$. Lemma \ref{projection rep random} gives us such a representation.

\begin{lemma} \label{projection rep random}
Let $X_n$ be a random vector in $\mathbb{R}^n$ with law $\mathbb{P}_{n,p,\mathbb{W}}$, and let $E_n$ be a random subspace of $\R^n$ distributed according to $\mu_{n,\lambda}$, independent of $X_n$, where $\lambda\in[0,1]$.
Then we have the distributional identity
\begin{align} \label{random identity}
|| P_{E_n}  X_n ||_2 \stackrel{\dint}{=} \frac{ \Big( \sum_{i=1}^n Z_i^2 \Big)^{1/2}    }{  \Big( \sum_{i=1}^n |Z_i|^p + W \Big)^{1/p} } \frac{\Big( \sum_{i=1}^n I^{(\lambda)}_i g_i^2 \Big)^{1/2} }{  \Big( \sum_{i=1}^n g_i^2 \Big)^{1/2} },
\end{align}
where $Z_i$ are independent $p$-Gaussians, $g_i$ are independent standard Gaussians, and $I^{(\lambda)}_i $ are independent Bernoulli random variables with probability $\lambda$ of taking the value $1$ and probability $1- \lambda$ of taking the value $0$. 
\end{lemma}

\begin{proof}
By definition, the random subspace measure $\mu_{n,\lambda}$ is equal to the random measure $\nu_{n, K_n}$, where $K_n$ is binomially distributed with parameters $n$ and $\lambda$. The result follows by noting that the if $K_n$ is independent of the $g_i$ and $Z_i$, and is binomially distributed with parameters $n$ and $\lambda$, then we have the equality in law
\begin{align*}
 \sum_{i=1}^nI^{(\lambda)}_i \, g_i^2  \stackrel{\dint}{=} \sum_{ i =1}^{K_n} g_i^2.
\end{align*}
Comparing \eqref{fixed identity} with \eqref{random identity}, this implies the result.
\end{proof} 

In the next section we introduce the Kolmogorov-Smirnov distance between random variables, and develop some first properties of this distance.

\subsection{Kolmogorov-Smirnov distances} \label{sec:KS}
Recall that the Kolmogorov-Smirnov distance between two real-valued random variables $X$ and $Y$ is given by 
\begin{align} \label{KS def}
\dint_{\mathsf{KS}}(X,Y) :=  \sup_{ t \in \mathbb{R}} \left| \P \left( X \leq t \right) - \mathbb{P} \left( Y \leq t \right) \right|.
\end{align}
In the cases that $F$ and $G$ are distribution functions associated with $X$ and $Y$, or $f$ and $g$ are the probability densities of $X$ and $Y$, we will abuse notation and we will write $\dint_{\mathsf{KS}}(F,G)$ and $\dint_{\mathsf{KS}}(f,g)$ for $\dint_{\mathsf{KS}}(X,Y)$, the Kolmogorov-Smirnov distance between $X$ and $Y$. (The meaning will always be clear from context.)

We now elaborate on some properties of Kolmogorov-Smirnov distance. First we note that if $X$ and $Y$ have continuous densities $f$ and $g$, then the supremum is attained at a point $t \in \mathbb{R}$ such that $f(t) = g(t)$. (There may be more than one such point.)

The total variation distance between real-valued random variables $X$ and $Y$ is defined as
\begin{align} \label{eq tv1}
\dint_{\mathsf{TV}}(X,Y) := \sup_{ A }  \big| \P \left( X \in A \right) - \mathbb{P} \left( Y \in A \right)  \big|,
\end{align}
where the supremum is taken over all Borel subsets of the real numbers. Again, we will write $\dint_{\mathsf{TV}}(F,G)$ and $\dint_{\mathsf{TV}}(f,g)$ for the total variation distance between distribution functions and densities. We will also use the fact below that the total variation distance between two densities $f$ and $g$ may also be written 
\begin{align*}
\dint_{\mathsf{TV}}(f,g) = \frac{1}{2} \int_{ - \infty}^\infty |f(s) - g(s) | \,\dint s. 
\end{align*}
In any case, by considering Borel subsets of the form $A_t := \{ s \in \mathbb{R} : s \leq t \}$ in \eqref{eq tv1} and comparing with \eqref{KS def}, it is plain that
\begin{align} \label{KSTV eq}
\dint_{\mathsf{KS}}(f,g)  \leq \dint_{\mathsf{TV}}(f,g).
\end{align}
Equality holds in \eqref{KSTV eq} whenever $f$ and $g$ are continuous and the equation $f(t) = g(t)$ has exactly one real solution. 

We also remark that the Kolmogorov-Smirnov distance is invariant under rescaling of the random variables. Namely, if $X$ and $Y$ are random variables and $\lambda$ is a non-zero real number, then
\begin{align} \label{berry scaling}
\dint_{\mathsf{KS}}( \lambda X, \lambda Y) = \dint_{\mathsf{KS}} ( X, Y) .
\end{align}
Moreover, it is immediate that Kolmogorov-Smirnov distances satisfy the triangle inequality
\begin{align} \label{berry triangle}
\dint_{\mathsf{KS}}(X,Z) \leq \dint_{\mathsf{KS}}(X,Y) + \dint_{\mathsf{KS}}(Y,Z).
\end{align}

\subsubsection{Kolmogorov-Smirnov distances between different Gaussian densities}

Now we take a moment to consider the Kolmogorov-Smirnov distances between centered Gaussian random variables with different variances $\sigma$ and $\tau$. We remark that by \eqref{berry scaling}, when $N_\sigma$ and $ N_\tau$ are Gaussian random variables of variance $\sigma^2$ and $\tau^2$ respectively, the quantity $\dint_{\mathsf{KS}}(N_\sigma, N_\tau)$ only depends on the ratio $\tau/\sigma$.

\begin{lemma} \label{gaussian berries}
For $\nu \in(0,\infty)$, let $f_\nu(s) := \frac{1}{ \sqrt{ 2 \pi \nu^2 } } e^{ - s^2 / 2 \nu^2}$, $s\in\R$. Then, whenever $\sigma < \tau$,
\begin{align*}
\dint_{\mathsf{KS}} ( f_\sigma, f_\tau)  \leq \frac{1}{4} (1 - \sigma/\tau) + \frac{1}{ 8} \left( \frac{ \tau^2 }{ \sigma^2}  - 1 \right).
\end{align*}
\end{lemma}
\begin{proof}
Since $\tau > \sigma$, we note that there is a $s_0 > 0$ such that $f_\tau (s) > f_\sigma(s)$ whenever $s < -s_0$ or $s > s_0$, and $f_\tau(s) < f_\sigma(s)$ whenever $-s_0<s<s_0$. In particular, 
\begin{align} \label{special prop}
\dint_{\mathsf{KS}}( f_\sigma, f_\tau) = \int_{-\infty}^{-s_0} f_\tau(s) - f_\sigma(s) \,\dint s =  \frac{1}{4} \int_{- \infty}^\infty \left| f_\sigma( s) - f_\tau(s) \right| \,\dint s.
\end{align}
Recall that $\sigma < \tau$. Using the triangle inequality to obtain the first inequality below, we have
\begin{align*}
&\left| \frac{1}{ \sqrt{ 2 \pi \sigma^2} }  e^{ - s^2 / 2 \sigma^2} - \frac{1}{ \sqrt{ 2 \pi \tau^2 } }  e^{ - s^2 / 2 \tau^2}\right|\\
&\leq \left|   \frac{1}{ \sqrt{2 \pi \sigma^2} } e^{ - s^2/ 2 \sigma^2 } - \frac{1}{ \sqrt{2 \pi \tau^2} } e^{ - s^2/ 2 \sigma^2 }    \right| +  \left|   \frac{1}{ \sqrt{2 \pi \tau^2} } e^{ - s^2/ 2 \sigma^2 } - \frac{1}{ \sqrt{ 2 \pi \tau^2 } }  e^{ - s^2 / 2 \tau^2}   \right|\\
&=  \frac{1}{ \sqrt{2 \pi \sigma^2} } e^{ - s^2/ 2 \sigma^2 }  \left( 1 - \frac{ \sigma}{ \tau} \right)  +  \frac{1}{ \sqrt{ 2 \pi \tau^2 } }  e^{ - s^2 / 2 \tau^2}   \left| e^{ - \frac{s^2}{2} \left( \frac{1}{\sigma^2} - \frac{1}{ \tau^2} \right) } - 1 \right|  \\
&\leq   \frac{1}{ \sqrt{2 \pi \sigma^2} } e^{ - s^2/ 2 \sigma^2 }  \left( 1 - \frac{ \sigma}{ \tau} \right) 
+  \frac{1}{ \sqrt{ 2 \pi \tau^2 } }  e^{ - s^2 / 2 \tau^2} \frac{s^2}{2} \left( \frac{1}{\sigma^2} - \frac{1}{ \tau^2} \right),
\end{align*}
where the final inequality above follows from  the fact that $| 1 - e^{ - \alpha}| < \alpha$ for positive $\alpha$. Integrating over $\mathbb{R}$, we obtain
\begin{align} \label{final1}
\int_{- \infty}^\infty \left| \frac{1}{ \sqrt{ 2 \pi \sigma^2} }  e^{ - s^2 / 2 \sigma^2} - \frac{1}{ \sqrt{ 2 \pi \tau^2 } }  e^{ - s^2 / 2 \tau^2}\right| \,\dint s \leq \left( 1 - \frac{ \sigma}{ \tau} \right)  + \frac{ \tau^2}{2} \left( \frac{1}{\sigma^2} - \frac{1}{ \tau^2} \right).
\end{align}
By combining \eqref{special prop} with \eqref{final1}, we obtain the result. 
\end{proof}

The following proposition states that Kolmogorov-Smirnov distance between different Gaussians is locally Lipschitz in the variance. A similar result stated in terms of total variation distance appears in Brehm and Voigt \cite[Theorem 3.1]{BV}, where a proof is given using the heat semigroup. 

\begin{proposition} \label{berry gauss lip}
For any $\alpha, \beta \in(0,\infty)$ such that $\frac{ \beta}{ \alpha} > \frac{1}{2}$, 
\begin{align*}
\dint_{\mathsf{KS}}( f_\alpha, f_\beta )  \leq \frac{3}{8} \,\frac{ | \alpha^2 - \beta^2|}{\alpha^2} \,.
\end{align*} 
\end{proposition}
\begin{proof}
If $\beta \geq \alpha$, then we may write $\beta = \alpha(1+\epsilon)$ for some suitable $\epsilon \geq 0$ so that $\frac{ | \alpha^2 - \beta^2|}{\alpha^2} = 2 \epsilon + \epsilon^2$. Then using Lemma \ref{gaussian berries} to obtain the first inequality below, and the fact that $\epsilon \geq 0$ to obtain the second, we have
\begin{align*}
\dint_{\mathsf{KS}}(f_\alpha,f_\beta) &\leq \frac{1}{4} \left( 1 - 1/(1+\epsilon) \right) + \frac{1}{8} \left( (1+\epsilon)^2 - 1 \right)\\
& \leq \frac{1}{4}  (2\epsilon + \epsilon^2)
 = \frac{1}{4}  \frac{ | \alpha^2 - \beta^2|}{\alpha^2} 
\leq \frac{3}{8}  \frac{ | \alpha^2 - \beta^2|}{\alpha^2}\,. 
\end{align*}
On the other hand, if $\beta < \alpha$, then we may write $\beta = \alpha( 1 - \epsilon)$ for some suitable $0 < \epsilon < \frac 1 2$ so that in this case $\frac{ | \alpha^2 - \beta^2|}{\alpha^2} = 2 \epsilon - \epsilon^2$. Again using Lemma \ref{gaussian berries} to obtain the first equality below, as well as the fact that $0< \epsilon < \frac{1}{2}$ to obtain the second inequality below, we have
\begin{align*}
\dint_{\mathsf{KS}}(  f_\alpha, f_\beta) &\leq \frac{1}{4} \big( 1 - (1-\epsilon) \big) + \frac{1}{8} \left( \frac{1}{( 1-\epsilon)^2 } - 1 \right)\\
&< \frac{3}{8} ( 2 \epsilon - \epsilon^2)  
 = \frac{3}{8}\frac{ | \alpha^2 - \beta^2|}{\alpha^2}\,.
\end{align*} 
This completes the proof.
\end{proof}

\subsubsection{The Berry-Esseen theorem}
Under certain conditions, the central limit theorem states that the sum of many independent random variables is approximately Gaussian. We are now ready to state the most prominent result using Kolmogorov-Smirnov distances, the Berry-Esseen theorem, which gives an upper bound on the Kolmogorov-Smirnov distance between a sum $S$ of independent centered random variables and a Gaussian random variable with the suitable variance. The version of the theorem we use, allowing the random variables occurring in the sum to be non-identically distributed, is due to Berry \cite{Berry}.

\begin{thm}[Berry-Esseen] \label{thm:berry esseen} 
There exists a universal constant $K_{\mathsf{BE}}\in(0,\infty)$ with the following property: for $m\in\N$ let $S_m$ be a sum of independent and centered random variables $X_1,\ldots,X_m$ with variances $\sigma_1^2,\ldots,\sigma_m^2$ and third order absolute moments $\rho_1,\ldots,\rho_m$. Let $G_m$ be a centered Gaussian random variable with the same variance as $S_m$. Then 
\begin{align*}
\dint_{\mathsf{KS}}(S_m,G_m) \leq K_{\mathsf{BE}} \frac{\max\limits_{ 1\leq i \leq m } \frac{\rho_i}{ \sigma_i^2}}{ \sqrt{ \sum_{i=1}^m\sigma_i^2 }}\,.
\end{align*}
In particular, if $X_1,\dots,X_m$ are identically distributed with variance $\sigma^2$ and third absolute moment $\rho$, then
\begin{align*}
\dint_{\mathsf{KS}}( S_m, G_m) \leq K_{\mathsf{BE}} \frac{ \rho}{ \sigma^3 \sqrt{m}}\,.
\end{align*}

\end{thm}
We remark that in Berry's original paper, $K_{\mathsf{BE}}$ may be taken to be $1.88$. This estimate has since been sharpened significantly, most recently by Shetsova \cite{She13} who showed the conclusions of Theorem \ref{thm:berry esseen} remain valid with $K_{\mathsf{BE}} = 0.5583$.

\subsection{Some further tools}

In Sections \ref{sec:RP 1 proof} and \ref{sec:q norm proof} we require the following lemma, lifted from \cite[Lemma 4.1]{APT2019}. 
\begin{lemma} \label{separating}
Let $X^{(1)},X^{(2)},X^{(3)}$ be (possibly dependent) random variables, and let $G$ be a centered Gaussian random variable with variance $\sigma^2$. Then, for any $\epsilon \in(0,\infty)$,
\begin{align*}
&\sup_{t \in \mathbb{R}} \big| \P[X^{(1)} + X^{(2)} + X^{(3)} \geq t ] - \P [ G \geq t] \big|\\
&\leq \sup_{t \in \mathbb{R} } | \P[ X^{(1)} \geq t ] - \P [ G \geq t ] | + \P \big[ |X^{(2)}| > \epsilon/2 \big] + \P \big[ |X^{(3)}| > \epsilon/2 \big] + \frac{ \epsilon}{ \sqrt{ 2 \pi \sigma^2 } } \,.
\end{align*}
\end{lemma}

We will use the Berry-Esseen theorem both directly and indirectly. In the latter case we have the following lemma, Lemma \ref{gaussish bound}, which is a generalization of an idea appearing at the bottom of page 21 in \cite{APT2019}. Since this lemma is rather technical, before giving a statement we would like to take a moment to provide some motivation. 

In our proofs of Theorems \ref{thm:RP 1} and \ref{thm:RP 2} we will encounter situations where for each $m \in\N$ we have a rescaled sum $S_m$ of $m$ i.i.d.~random variables whose distributions may depend on $m$. Given a sequence $(\beta_m)$ of reals tending to infinity, we would like the find the \emph{smallest} sequence $(\alpha_m)$ ensuring that $\P[ | S_m | > \alpha_m] \leq C/ \sqrt{\beta_m}$ for some constant $C\in(0,\infty)$ not depending on $m$. Lemma \ref{gaussish bound} states the best possible choice is to take $\alpha_m$ of the order $\sqrt{ \log ( \beta_m) }$.

In the proof of Lemma \ref{gaussish bound} we will use the notation 
\begin{align*}
\sqrt{a}_+ := \begin{cases}
\sqrt{a}, ~~ &\text{if $a \geq 0$,}\\
0, ~~ &\text{if $a < 0$},
\end{cases}
\end{align*}
as well as the following well known inequality for Gaussian integrals, 
which, for any $t>0$, provides the estimate
\begin{align} \label{tail bound}
\int_t^\infty \frac{1}{ \sqrt{ 2 \pi }} e^{ - u^2/2} \,\dint u \leq \frac{1}{ t} \frac{1}{ \sqrt{2 \pi }} e^{ - t^2/2}.
\end{align}

\begin{lemma} \label{gaussish bound}
Let $X^{(m)}$, $m\in\N$ be a sequence of centered random variables with $\E[ | X^{(m)}|^3 ] = \rho_m$ and $\E[ |X^{(m)}|^2 ] = \sigma_m^2$. Suppose $\sigma_{\mathsf{max}} \geq \sigma_m$ for every $m\in\N$, and that $\gamma_m$, $m\in\N$ is any sequence satisfying 
\begin{align} \label{cmax}
\frac{\rho_m}{\sigma_m^{3}} \leq \gamma_m \,.
\end{align}
Suppose for each $m\in\N$ we have a sum
\begin{align*}
S_m = \frac{1}{ \sqrt{m}} \left( X_1^{(m)} + \ldots + X_m^{(m)} \right),
\end{align*}
where $X_1^{(m)},\dots, X_m^{(m)}$ are independent copies of $X^{(m)}$. Suppose $\beta_m$, $m\in\N$ is a sequence of positive reals such that $\beta_m \leq m / \gamma_m^2$ for every $m\in\N$. Then, for any sequence $(\alpha_m)_{m \geq 1}$ of positive reals satisfying  $\alpha_m \geq \sigma_{\mathsf{max}} \sqrt{ \log (\beta_m) }_+$ for every $m\in\N$, we have
\begin{align} \label{sum bound}
\P \big[|S_m| > \alpha_m \big] \leq \frac{C^* }{ \sqrt{\beta_m }} \quad\text{ for every $m\in\N$,}
\end{align}
for $C^* =  2 \left( K_{\mathsf{BE}}   + 1 \right) $.
\end{lemma}

\begin{proof}
Since $C^* \geq 2$, the bound \eqref{sum bound} holds trivially when $\beta_m \leq 4$, and we may assume without loss of generality that $\beta_m > 4$ for the remainder of the proof. Let $G_m$ be a centered Gaussian random variable with variance $\sigma_m^2$. Then using the definition of Kolmogorov-Smirnov distance and the symmetry of the Gaussian density, for any $\alpha_m >0 $ we have the inequality
\begin{align} \label{bound4}
\P \big[ |S_m| > \alpha_m \big] \leq 2\, \dint_{\mathsf{KS}} \left( S_m, G_m \right) + 2\, \P \big[ G_m > \alpha_m \big].
\end{align}
Indeed, it holds that
\begin{align*}
\Pro[|S_m|>\alpha_m] & = \Pro[|S_m|>\alpha_m] - \Pro[|G_m|>\alpha_m] + \Pro[|G_m|>\alpha_m] \cr
& = \Pro[|S_m|>\alpha_m] - \Pro[|G_m|>\alpha_m] +2\, \Pro[G_m>\alpha_m]
\end{align*}
and since 
\[
\Pro[|S_m|>\alpha_m] = 1-\Big(\Pro[S_m \leq \alpha_m] - \Pro[S_m \leq -\alpha_m] \Big)
\]
as well as 
\[
\Pro[|G_m|>\alpha_m] = 1-\Big(\Pro[G_m \leq \alpha_m] - \Pro[G_m \leq -\alpha_m] \Big),
\]
we obtain
\begin{align*}
\Pro[|S_m|>\alpha_m] - \Pro[|G_m|>\alpha_m] & =  \Pro[G_m \leq \alpha_m] - \Pro[S_m \leq \alpha_m] + \Pro[S_m \leq -\alpha_m] - \Pro[G_m \leq -\alpha_m] \cr
& \leq 2\, \dint_{\mathsf{KS}} \left( S_m, G_m \right).
\end{align*}
Now, on the one hand, using the Berry-Esseen theorem to obtain the first inequality below,  the estimate \eqref{cmax} to obtain the second, and the fact that $\beta_m \leq m/\gamma_m^2$ to obtain the third, we have
\begin{align} \label{bound2}
 \dint_{\mathsf{KS}} \left( S_m, G_m \right) \leq \frac{  K_{\mathsf{BE}} \rho_m/\sigma_m^{3}}{ \sqrt{m}  }  \leq \frac{ K_{\mathsf{BE}} \gamma_m }{ \sqrt{m}} \leq \frac{ K_{\mathsf{BE}}  }{ \sqrt{\beta_m } }.
\end{align}
On the other hand, $G_m/\sigma_m$ is a standard Gaussian random variable, and using the inequality \eqref{tail bound} we have
\begin{align} \label{bound1}
  \P \big[ G_m > \alpha_m \big] = \P \bigg[ \frac{G_m}{ \sigma_m} > \frac{ \alpha_m}{ \sigma_m } \bigg] \leq \frac{ 1}{ \alpha_m/ \sigma_m } \frac{1}{ \sqrt{2 \pi }} e^{ - \frac{1}{2} \left( \frac{ \alpha_m}{ \sigma_m } \right)^2 },
\end{align} 
and plugging $\alpha_m \geq \sigma_{\mathsf{max}} \sqrt{ \log ( \beta_m) }_{+} \geq  \sigma_m \sqrt{ \log ( \beta_m) }_{+} $ into \eqref{bound1} we have
\begin{align} \label{bound3}
  \P \big[ G_m > \alpha_m \big] \leq \frac{1}{ \sqrt{2 \pi }} \frac{1}{ \sqrt{ \beta_m} \sqrt{\log (\beta_m) }_{+} } \leq  \frac{1}{ \sqrt{\beta_m}}.
\end{align}
where in the final inequality above we used the fact that $\beta_m > 4 > \exp(1)$ and the crude bound $\frac{1}{ \sqrt{2 \pi}} \leq 1$. Combining \eqref{bound2} and \eqref{bound3} in \eqref{bound4}, and using the definition of $C^*$, the result follows. 
\end{proof}

\subsubsection{Bounds for second order functions}

Our final tool provides us with a bound on the probability that a function with local second order growth takes on values bounded away from $0$.

\begin{lemma}  \label{lem:Psi bound}
Suppose that $v$ is a random vector in $\R^d$ and $\Psi: \mathbb{R}^d \to \mathbb{R}$ is a function with the property that there exist $\delta, M \in(0,\infty)$ such that
\begin{align} \label{property}
|\Psi(x) | \leq M || x||_2^2 \quad \text{ whenever }\quad ||x||_2 < \delta.
\end{align}
Then, for any $c \in(0,\infty)$, 
\begin{align*}
\P \big[ | \Psi(v) | > c  \big]  \leq \P \bigg[ || v||_2 >  \sqrt{ \frac{ c}{ M }}  \,\bigg]  + \P \big[ ||v||_2 \geq \delta \big].
\end{align*}
\end{lemma}
\begin{proof}
Using the property \eqref{property} to obtain the second inequality below, we have
\begin{align*}
\P \big[ | \Psi(v) | > c  \big]  &= \P \big[ | \Psi(v) | > c, ||v||_2 < \delta \big]  + \P \big[ | \Psi(v) | > c  , ||v||_2 \geq \delta \big] \\
&\leq \P \big[ M|| v||_2^2 > c, ||v||_2 < \delta \big]  + \P \big[ | \Psi(v) | > c  , ||v||_2 \geq \delta \big] \\
&\leq \P \big[ M|| v||^2_2 >c  \big]  + \P \big[ ||v||_2 \geq \delta \big].
\end{align*}
This establishes the result.
\end{proof}

\section{Proof of Theorem \ref{thm:RP 1}}\label{sec:RP 1 proof}

In this section we prove Theorem \ref{thm:RP 1}, which establishes Berry-Esseen bounds for the $\ell_2$-norm of the rescaled projection of a random vector with distribution $\mathbb{P}_{n,p,\mathbb{W}}$ onto a random $k$-dimensional subspace of $\R^n$ distributed according to $\nu_{n,k}$.  
\subsection{The set up} \label{sec:RP 1 setup}
Let $X_n$ be a random vector with law $\mathbb{P}_{n,p,\mathbb{W}}$, and let  $E_n$ be a random subspace of $\mathbb{R}^n$ with law $\nu_{n,k_n}$, $k_n\in\{1,\dots,n\}$. Then by Proposition \ref{projection rep} we have the identity in law
\begin{align*}
|| P_{E_n} X_n ||_2 \stackrel{\dint}{=} \frac{ \Big( \sum_{ i =1}^n Z_i^2 \Big)^{1/2} \Big( \sum_{ i =1}^{k_n} g_i^2 \Big)^{1/2} }{ \Big( \sum_{ i =1}^n |Z_i|^p + W \Big)^{1/p} \Big( \sum_{ i =1}^n g_i^2 \Big)^{1/2}  }\,,
\end{align*}
where $Z_i$ be a collection of independent $p$-Gaussian random variables, $W$ has law $\mathbb{W}$, and $g_i$ are independent standard Gaussian random variables. 

Shifting the sums to have expectation zero, we may write
\begin{align}
|| P_{E_n} X_n ||_2 &\stackrel{\dint}{=}  \frac{ \left( n M_p(2) + \sum_{ i =1}^n (Z_i^2  - M_p(2)) \right)^{1/2} \left( k_n +  \sum_{ i =1}^{k_n}  (g_i^2 - 1 )  \right)^{1/2} }{ \left( n + \sum_{ i =1}^n ( |Z_i|^p - 1)  + W \right)^{1/p} \left( n +  \sum_{ i =1}^n (g_i^2 - 1)  \right)^{1/2}  }, \nonumber \\
&= \frac{ \sqrt{k_n}  \sqrt{M_p(2)} }{ n^{1/p} } F\left( \frac{a_n}{n M_p(2) }, \frac{b_n}{k_n}, \frac{c_n}{n }, \frac{ W}{ n} , \frac{ d_n}{ n} \right), \label{exp11}
\end{align}
where $F(a,b,c,w,d) = \frac{ (1+a)^{1/2} (1 + b)^{1/2}  }{ ( 1 + c + w)^{1/p} (1 + d)^{1/2} }$, and 
\begin{align}
a_n = \sum_{i  =1}^n (Z_i^2 - M_p(2)), \qquad  b_n = \sum_{i  =1}^{k_n}  (g_i^2 - 1 ), \nonumber \\
c_n = \sum_{i  =1}^n (|Z_i|^p - 1), \qquad  d_n = \sum_{i  =1}^n (g_i^2 - 1) .  \label{abcd def}
\end{align}
The function $F: \mathbb{R}^5 \to \mathbb{R}$ is twice differentiable in a neighborhood around the origin in $\mathbb{R}^5$, and hence, by Taylor's theorem we may write
\begin{align} \label{F Psi}
F(a,b,c,w,d) = 1 + \frac{1}{2} ( a + b - d) - \frac{1}{p} (c + w ) + \Psi(a,b,c,w,d),
\end{align}
where $\Psi: \mathbb{R}^5 \to \mathbb{R}$ is a function such that there exist $M,\delta \in(0,\infty)$ satisfying
\begin{align} \label{Psi bound 2}
| \Psi(x) | \leq M ||x||_2^2 \quad \text{ whenever } \quad ||x||_2 \leq \delta. 
\end{align}
It follows from plugging \eqref{F Psi} into \eqref{exp11} that the focal random variable $Y_n$ may be decomposed as follows:
\begin{align}
Y_n &:= \frac{ n^{1/p}  }{ \sqrt{M_p(2)}} || P_{E_n} X_n||_2 - \sqrt{k_n} \nonumber \\
&\stackrel{\dint}{=} \xi_n  - \frac{ \sqrt{k_n} W }{p n} + \sqrt{k_n} \Psi \left( \frac{a_n}{n M_p(2) }, \frac{b_n}{k_n}, \frac{c_n}{n }, \frac{ W}{ n } , \frac{ d_n}{ n} \right), \label{exp 1}
\end{align}
where $\xi_n := \sqrt{k_n} \left( \frac{1}{ 2 n M_p(2)} a_n + \frac{1}{2 k_n} b_n - \frac{1}{p n} c_n - \frac{1}{ 2 n} d_n \right)$. 

Assuming that $\lambda_n := k_n/n$ converges to some $\lambda \in [0,1]$, we recall that the goal of this proof is to establish a bound for the Kolmogorov-Smirnov distance between $Y_n$ and a Gaussian random variable $G$ with variance $v(\lambda)$, where $v:[0,1] \to \mathbb{R}$ is given by
\begin{align} \label{v def}
v(s)  := s \sigma^2(p,2) + \frac{1}{2}(1-s).
\end{align}
 The main idea of the proof is that we may apply the Berry-Esseen theorem to show that $\xi_n$, which is a sum of many random variables, is close to a Gaussian random variable, and the remaining terms in the expansion \eqref{exp 1} for $Y_n$ are negligible.

To this end, using Lemma \ref{separating} to obtain the first inequality, and the triangle inequality \eqref{berry triangle}  for Kolmogorov-Smirnov distances to obtain the second, for any $\epsilon > 0$ we have
\begin{align}
\dint_{\mathsf{KS}}(Y_n, G) & \leq \dint_{\mathsf{KS}}(\xi_n, G) \nonumber  
 + \P \left[  \frac{ W }{p n}  > \frac{ \epsilon}{ 2 \sqrt{k_n}} \right] \cr
 &\qquad+  \P \left[ \left| \Psi \left( \frac{a_n}{n M_p(2) }, \frac{b_n}{k_n}, \frac{c_n}{n }, \frac{ W}{ n} , \frac{ d_n}{ n} \right) \right| > \frac{ \epsilon}{ 2 \sqrt{k_n}}  \right] 
  + \frac{ \epsilon}{ \sqrt{2 \pi v(\lambda) } } \nonumber  \\
& \leq \dint_{\mathsf{KS}}(\xi_n, G_n)  + \dint_{\mathsf{KS}}(G_n, G)  
+ \P \left[  \frac{ W }{p n}  >\frac{ \epsilon}{ 2 \sqrt{k_n}} \right] \cr 
&\qquad +  \P \left[ \left| \Psi \left( \frac{a_n}{n M_p(2) }, \frac{b_n}{k_n}, \frac{c_n}{n }, \frac{ W}{ n} , \frac{ d_n}{ n} \right) \right| > \frac{ \epsilon}{ 2 \sqrt{k_n}}  \right] + \frac{ \epsilon}{ \sqrt{2 \pi J_p } } \label{expansion},
\end{align}
where $G_n$ is a Gaussian random variable with the same variance as $Y_n$, and 
\begin{align} \label{J_p def}
 J_p & := \inf_{s \in [0,1]} v(s)  = \min \big\{ \sigma^2(p,2) , 1/2 \big\} 
\end{align}
is strictly positive. 

The remainder of the proof is structured as follows. First we use the Berry-Esseen theorem to find a bound for $\dint_{\mathsf{KS}}(\xi_n,G_n)$. Then we use the results of Section \ref{sec:KS} on the Kolmogorov-Smirnov distances between different Gaussian distributions to bound $\dint_{\mathsf{KS}}(G_n,G)$. After this we take a choice $\epsilon = \epsilon_n$ to establish a bound for the term involving $\Psi$. Finally, we take stock of these inequalities to give a proof of Theorem \ref{thm:RP 1}.

\subsection{A bound for $\dint_{\mathsf{KS}}(\xi_n,G_n)$}

We note that the constituent parts $a_n ,b_n,c_n, d_n$ of the random variable $\xi_n$ are clearly not independent of one another. However, $\xi_n$ may be written as a sum of $n + k_n + (n-k_n)$ independent random variables by writing
\begin{align*}
\xi_n = & \sum_{ i = 1}^n   \frac{ \sqrt{\lambda_n } \varphi_i }{ \sqrt{n}} +  \sum_{ i = 1}^{k_n }  \frac{1 - \lambda_n }{2 \sqrt{k_n}} \phi_i + \sum_{ i = k_n + 1}^n \left( - \frac{ \sqrt{ \lambda_n(1 - \lambda_n) } }{2 \sqrt{ n - k_n} } \right) \phi_i,
\end{align*}
where we recall $\lambda_n := k_n/n$, and define for $i\in\{1,\dots,n\}$ the random variables
\begin{align*}
\varphi_i :=   \frac{ |Z_i|^2 - M_p(2) }{ 2 M_p(2) } - \frac{ |Z_i|^p - 1}{ p  } \qquad\text{and}\qquad \phi_i := g_i^2 - 1  \,.
\end{align*}
Let $\varphi$ and $\phi$ be random variables with the same law as $\varphi_1$ and $g_1$ respectively. Then noting that the variance of $\varphi$ is given by $\sigma^2(p,2)$ (see the display right after Lemma \ref{p Gaussian moments}) and the variance of $\phi$ is simply $2$, we obtain that
\[
\Var[\xi_n] = \lambda_n\Var[\varphi]+ \Big(\frac{1-\lambda_n}{2}\Big)^2 \Var[\phi] + \frac{\lambda_n(1-\lambda_n)}{4}\Var[\phi] = \lambda_n\sigma^2(p,2) + \frac{1}{2}(1-\lambda_n) = v(\lambda_n),
\]
where $v(s)$ is given by \eqref{v def}. 

Now let $G_n$ be a centered Gaussian random variable with variance $v(\lambda_n)$. Then in the context of the Berry-Esseen Theorem (Theorem \ref{thm:berry esseen}), with $m = n + k_n + (n - k_n)$, we have
\begin{align*}
\max_{ 1\leq  i \leq m} \frac{\rho_i}{\sigma_i^2}  &= \max \left\{ \frac{ \E \left[  \left| \frac{ \sqrt{\lambda_n } \varphi }{ \sqrt{n}} \right|^3 \right] }{ \E \left[ \left|  \frac{ \sqrt{\lambda_n } \varphi }{ \sqrt{n}} \right|^2 \right] },  \frac{ \E \left[ \left| \frac{1 - \lambda_n }{2 \sqrt{k_n}} \phi  \right|^3 \right] }{ \E \left[ \left| \frac{1 - \lambda_n }{2 \sqrt{k_n}} \phi  \right|^2 \right] } ,  \frac{ \E \left[ \left|  \left( - \frac{ \sqrt{ \lambda_n(1 - \lambda_n) } }{2 \sqrt{ n - k_n} } \right) \phi  \right|^3 \right] }{ \E \left[ \left| \left( - \frac{ \sqrt{ \lambda_n(1 - \lambda_n) } }{2 \sqrt{ n - k_n} } \right) \phi  \right|^2 \right] } \right\} \leq C_p A_n.
\end{align*}
where
\begin{align*}
C_p  :=  \max \left\{ \frac{ \E[ | \varphi|^3 ] }{ \E[ \varphi^2] }, \frac{ \E[ | \phi|^3 ] }{ \E[ \phi^2] } \right\} \qquad \text{and} \qquad A_n := \max \left\{   \frac{\sqrt{\lambda_n}}{ \sqrt{n}}, \frac{ 1 - \lambda_n}{ 2 \sqrt{k_n } } , \frac{ \sqrt{ \lambda_n (1- \lambda_n) }}{ 2 \sqrt{n-k_n}}    \right\} .
\end{align*}
Using the fact that $\lambda_n \leq 1$ and $\lambda_n := k_n/n$, it follows that $A_n \leq \frac{1}{ \sqrt{k_n}}$. Consequently, 
\[
\max_{1 \leq i \leq m } \frac{\rho_i}{\sigma_i^2} \leq \frac{C_p}{\sqrt{k_n}},
\] 
and by applying the Berry-Esseen theorem to the sequence $\xi_n$, $n\in\N$ we have
\begin{align} \label{5a}
\dint_{\mathsf{KS}}( \xi_n, G_n ) \leq K_{\mathsf{BE}} \frac{ C_p}{\sqrt{v(\lambda_n)}  \sqrt{k_n}  } \leq K_{\mathsf{BE}} \frac{ C_p }{ \sqrt{J_p} \sqrt{k_n}} =: \frac{C_p'}{ \sqrt{k_n}},
\end{align}
where the second inequality above follows from using the definition \eqref{J_p def} of $J_p$.

\subsection{A bound for $\dint_{\mathsf{KS}}(G_n,G)$}
The random variables $G_n$ and $G$ are centered Gaussians with variances given by $v(\lambda_n)$ and $v(\lambda)$ respectively.  It follows from Proposition \ref{berry gauss lip} and the definition of $J_p$ in \eqref{J_p def} that
\begin{align} \label{5b}
\dint_{\mathsf{KS}}(G_n,G) \leq \frac{3}{8} \frac{ |v(\lambda_n) - v(\lambda) | }{ v(\lambda) } \leq \frac{3}{8} \frac{| \sigma^2(p,2) -\frac{1}{2} | | \lambda_n - \lambda|  }{ J_p } = C_p\, | \lambda_n - \lambda|,
\end{align}
where
\[
C_p := \frac{3}{8}\,\frac{|\sigma^2(p,2)-\frac{1}{2}|}{\min\{\sigma^2(p,2),\frac{1}{2}\}}\,.
\]

\subsection{A bound for the $\Psi$ term} \label{sec:RP 1 Psi bound}
In this section we take a sequence $(\epsilon_n)_{n\in\N}$ decreasing down to $0$ in such a way that the probability
\begin{align*}
\P \left[ \left| \Psi \left( \frac{a_n}{n M_p(2) }, \frac{b_n}{k_n}, \frac{c_n}{n }, \frac{ W}{ n } , \frac{ d_n}{ n} \right) \right| >\frac{ \epsilon_n}{ 2 \sqrt{k_n}}  \right] 
\end{align*}
may be bounded by a function of $n$ decreasing down to $0$ as $n \to \infty$.

To this end, using Lemma \ref{lem:Psi bound} to obtain the first inequality below, and the triangle inequality to obtain the second, with $v_n := \left( \frac{a_n}{n M_p(2) }, \frac{b_n}{k_n}, \frac{c_n}{n }, \frac{ W}{ n} , \frac{ d_n}{ n} \right)  \in \mathbb{R}^5$ we have  
\begin{align}
& \P \left[ \left| \Psi \left( v_n\right) \right| > \frac{ \epsilon_n }{ 2 \sqrt{k_n} }  \right]  \leq  \P \left[ \left| \left| v_n \right|  \right|_2 \geq  \sqrt{ \frac{ \epsilon_n }{ 2M \sqrt{k_n} } }  \right] + 
\P \big[ \left| \left|   v_n \right|  \right|_2 \geq  \delta \big] \nonumber \\
&\leq Q_n \left( \frac{a_n}{ n M_p(2) } \right) +  Q_n \left( \frac{b_n}{k_n} \right) +   Q_n \left(  \frac{c_n}{n } \right) +  Q_n \left( \frac{ W}{ n } \right)  +  Q_n \left(  \frac{ d_n}{ n} \right),  \label{expansion 2}
\end{align} 
where for a random variable $X$,
\begin{align*}
Q_n(X) := \P \left[ |X| >   \sqrt{ \frac{ \epsilon_n }{ 10M \sqrt{k_n} } } \right ]+   \P \left[ |X| >  \frac{\delta}{\sqrt{5}} \right] .
\end{align*}
(We emphasize that $Q_n$ depends on $\epsilon_n$.) We now show that with a suitable choice of $\epsilon_n$, there is a constant $C\in(0,\infty)$ such that the terms involving $a_n, b_n, c_n, d_n$ in the final expression in \eqref{expansion 2} may be bounded above by $C/ \sqrt{k_n}$. 

First consider obtaining an upper bound for 
\begin{align*}
Q_n \left( \frac{a_n}{ n M_p(2) } \right) = \P \left[ \left| \frac{a_n}{ \sqrt{n}} \right| > c_p \sqrt{n} \right] + \P \left[ \left| \frac{a_n}{ \sqrt{n}} \right| >  c_p' \sqrt{ \frac{ n \epsilon_n }{ \sqrt{k_n} } }   \right]  
\end{align*}
where $c_p = \frac{ \delta M_p(2)}{ \sqrt{5}}$ and $c_p' = \frac{M_p(2)}{ \sqrt{10 M }}$.

Letting $S_n := \frac{ a_n }{ \sqrt{n}}$, $n\in\N$ we are in the set up of Lemma \ref{gaussish bound} with $m=n$, and with two different tail probabilities we would like to bound, namely
\begin{align*}
\alpha_n := c_p \sqrt{n} \qquad \text{and} \qquad \alpha'_n := c_p' \sqrt{ \frac{ n \epsilon_n }{ \sqrt{k_n} } }  .
\end{align*}
Since $a_n$ is a sum of $n$ random variables with distribution $Z^2 - M_p(2)$ (where $Z$ is $p$-Gaussian), in the setting of Lemma \ref{gaussish bound} we may take 
\begin{align*}
\gamma_n := \frac{ \E\big[|Z_1^2 - M_p(2)|^3 \big] }{ \E\big[|Z_1^2 - M_p(2)|^2 \big]} \qquad \text{for every $n\in\N$},
\end{align*} 
and $\sigma_{\mathsf{max},p} = \E[|Z_1^2 - M_p(2)|^2 ]$. Indeed, whenever a sequence $\beta_n$ and $\epsilon_n$ are chosen so that the inequalities
\begin{align} \label{pair ineq}
c_p \sqrt{n}  \geq \sigma_{\mathsf{max},p}  \sqrt{ \log(\beta_n)}_+  \qquad \text{and} \qquad  c_p' \sqrt{ \frac{ n \epsilon_n }{ \sqrt{k_n} } } \geq \sigma_{\mathsf{max},p}\sqrt{ \log(\beta_n)}_+
\end{align}
hold for every $n\in\N$, then by Lemma \ref{gaussish bound} we have 
\begin{align} \label{solo ineq}
Q_n\left( \frac{a_n}{ n M_p(2) } \right) \leq \frac{ 4 \left( K_{\mathsf{BE}} + 1 \right)}{ \sqrt{\beta_n}},
\end{align}
for every $n\in\N$.
Now consider setting $\epsilon_n = \Theta_p \log(k_n)/ \sqrt{k_n}$ and $\beta_n = \theta_p k_n$. Then if $\theta_p$ is sufficiently small and $\Theta_p$ is sufficiently large such that both the inequalities in \eqref{pair ineq} hold, then \eqref{solo ineq} holds, and there is a constant $C_p'\in(0,\infty)$ such that
\begin{align*}
Q_n\left( \frac{a_n}{ n M_p(2) } \right) \leq \frac{C_p'}{\sqrt{k_n}}
\end{align*}
for every $n\in\N$. 
We may work analogously with the terms $Q_n\left( \frac{c_n}{n} \right)$ and $Q_n \left( \frac{d_n}{ n } \right)$, which also involve tail bounds for sums of $n$ independent random variables. It follows that for sufficiently large constants $C_p''$ and $\Theta_p'$, by setting $\epsilon_n = \Theta_p' \frac{ \log(k_n)}{ \sqrt{k_n}}$ we ensure that 
\begin{align} \label{5a}
Q_n\left( \frac{a_n}{ n M_p(2) } \right) + Q_n\left( \frac{c_n}{n} \right) + Q_n \left( \frac{d_n}{ n } \right) \leq \frac{C_p''}{\sqrt{k_n}}
\end{align}
for every $n\in\N$. 

We now consider obtaining an upper bound for $Q_n \left( \frac{b_n}{ k_n} \right)$, which this time involves finding tail probabilities for a sum of $k_n$ random variables. We may write
\begin{align*}
Q_n \left( \frac{b_n}{ k_n} \right) = \P \left( \left| \frac{b_n}{ \sqrt{k_n}} \right| > c_p'' \sqrt{k_n} \right) + \P \left( \left| \frac{b_n}{ \sqrt{k_n}} \right| > c_p''' \sqrt{ \epsilon_n \sqrt{k_n}} \right) ,
\end{align*}
where $c_p'' = \frac{ \delta}{ \sqrt{5}}$ and $c_p''' = \frac{1}{ \sqrt{10 M }}$. Again we are in the context of Lemma \ref{gaussish bound}, but this time with $m = k_n$ and 
\begin{align*}
\alpha_m'' = c_p'' \sqrt{k_n} \qquad \text{and} \qquad \alpha_m''' = c_p''' \sqrt{ \epsilon_n \sqrt{k_n}}.
\end{align*} 
(We will continue to index all variables here with a subscript $n$ as opposed to $m$.) 
Now $b_n$ is a sum of $k_n$ random variables distributed like $\phi = g^2  -1$, where $g$ is a standard Gaussian. This time take $\gamma_n'  := \E[ |\phi|^3]/\E[ |\phi|^2]$ for every $n\in\N$, as well as $\sigma'_{\mathsf{max},p} := \E[ |\phi|^2]$. Then by Lemma \ref{gaussish bound}, whenever $\beta_n \leq n/ \gamma_n^2$ and $\epsilon_n$ are chosen so that the inequalities
\begin{align} \label{pair ineq'}
 c_p'' \sqrt{k_n} \geq \sigma'_{\mathsf{max},p}  \sqrt{ \log(\beta_n)}_+  \qquad \text{and} \qquad c_p''' \sqrt{ \epsilon_n \sqrt{k_n}  } \geq \sigma'_{\mathsf{max},p}  \sqrt{ \log(\beta_n)}_+ 
\end{align}
hold for every $n\in\N$, we have
\begin{align} \label{solo ineq'}
Q_n \left( \frac{b_n}{ k_n} \right)  \leq \frac{ 4 \left( K_{\mathsf{BE}} +  1 \right) }{  \sqrt{\beta_n  } } .
\end{align}
In particular, setting $\beta_n = \theta_p' k_n$ and $\epsilon_n := \Theta_p' \log(k_n)/ \sqrt{k_n}$, for sufficiently small $\theta_p'$ and sufficiently large $\Theta_p'$ the inequalities \eqref{pair ineq'} hold for every $n\in\N$. It follows that by \eqref{solo ineq'} there is a constant $C_p'''\in(0,\infty)$ such that
\begin{align} \label{5b}
Q_n \left( \frac{b_n}{ k_n} \right)  \leq  \frac{ C_p''}{ \sqrt{k_n}}.
\end{align} 
In particular, by setting $\epsilon_n := \Theta_p^* \log(k_n)/ \sqrt{k_n}$ with $\Theta_p^* := \max \{ \Theta_p , \Theta_p' \},$ both \eqref{5a} and \eqref{5b} hold for every $n\in\N$, and hence there is a constants $C_p^*\in (0,\infty)$ such that 
\begin{align} \label{penult}
Q_n \left( \frac{a_n}{ n M_p(2) } \right) +  Q_n \left( \frac{b_n}{k_n} \right) +   Q_n \left(  \frac{c_n}{n } \right)  +  Q_n \left(  \frac{ d_n}{ n} \right) \leq \frac{C_p^*}{\sqrt{k_n}}.
\end{align} 
In particular, by plugging \eqref{penult} into \eqref{expansion 2}, there are constants $\Theta_p^*, C_p^*\in(0,\infty)$ depending only on $p$ such that with the choice $\epsilon_n := \Theta_p^* \frac{ \log(k_n)}{ \sqrt{ k_n } }$ we have
\begin{align} \label{final boy}
\P \left[ \left| \Psi \left( \frac{a_n}{n M_p(2) }, \frac{b_n}{k_n}, \frac{c_n}{n }, \frac{ W}{ n } , \frac{ d_n}{ n} \right) \right| >\frac{ \epsilon_n}{ 2 \sqrt{k_n}}  \right]  & \leq \frac{C_p^*}{ \sqrt{k_n}} + \P \left[ W > n \sqrt{\frac{ \epsilon_n}{ 10 M \sqrt{k_n}}}  \right] \cr 
& \quad+ \P \left[ W > \frac{\delta}{\sqrt{5}}\, n \right].
\end{align}

\subsection{Wrapping things together}

Recall our original bound \eqref{expansion}. We obtained the bounds \eqref{5a} and \eqref{5b} for the $\epsilon_n$-independent terms. Then in Section \ref{sec:RP 1 Psi bound} we showed there are constants $C_p^*, \Theta_p^*\in(0,\infty)$ such that if $\epsilon_n := \Theta^*_p \frac{ \log(k_n) }{ \sqrt{k_n}}$, then \eqref{final boy} holds. Consider the collection of terms involving $W$ that occur in \eqref{expansion} and \eqref{final boy}. Noting that by definition $W$ is non-negative, with  $\epsilon_n := \Theta^*_p \frac{ \log(k_n) }{ \sqrt{k_n}}$  it follows that there is a constant $c_p\in(0,\infty)$ such that the $W$ terms may be bounded by 
\begin{align} \label{W bound}
 & \P \left[ W >  n \sqrt{ \frac{ \epsilon_n }{ 10M \sqrt{k_n} } } \right ]+   \P \left[ W >  \frac{\delta}{\sqrt{5}} n \right]  +  \P \left[  \frac{ W }{p n}  >\frac{ \epsilon_n }{ 2 \sqrt{k_n}} \right]  \cr
 & \qquad \leq 3\, \P \left[ W > c_p n \frac{ \log(k_n)}{k_n } \right] \,.
\end{align}
Plugging \eqref{5a}, \eqref{5b}, \eqref{final boy} into \eqref{expansion}, and using $\epsilon_n := \Theta_p^* \log(k_n) / \sqrt{k_n}$, and \eqref{W bound}, it follows that there exist constants $c_p,C_p\in(0,\infty)$ depending on $p$ but independent of $\lambda_n, \lambda$, and $n$ such that
\begin{align*}
\dint_{\mathsf{KS}} (Y_n, G) \leq C_p \max  \left\{ \frac{\log (k_n)}{ \sqrt{k_n} }, | \lambda_n - \lambda| \right\} + 3 \, \P \left( W > c_p n \frac{ \log(k_n)}{k_n } \right),
\end{align*}
which completes the proof of Theorem \ref{thm:RP 1}.

\section{Proof of Theorem \ref{thm:RP 2}}  \label{sec:RP 2 proof}
In this section we prove Theorem \ref{thm:RP 2}, which gives Berry-Esseen bounds for the rescaled $\ell_2$-norm of the orthogonal projection of a random vector with distribution $\mathbb{P}_{n,p,\mathbb{W}}$ onto a $\mu_{n,\lambda_n}$-distributed random subspace of $\mathbb{R}^n$. In fact, in this setting $ n \lambda_n$ has the significance of being the average dimension of a subspace with distribution $\mu_{n ,\lambda_n}$.

\subsection{Setting up the proof}
Let $X_n$ be a random vector with law $\mathbb{P}_{n,p,\mathbb{W}}$, and let  $E_n$ be a random subspace of $\mathbb{R}^n$ with law $\mu_{n,\lambda_n}$. Then by Lemma \ref{projection rep random} we have the identity in law
\begin{align*}
|| P_{E_n} X_n ||_2 \stackrel{\dint}{=} \frac{ \Big( \sum_{ i =1}^n Z_i^2 \Big)^{1/2} \Big( \sum_{ i =1}^{n} I^{(\lambda_n)}_{i}\, g_i^2 \Big)^{1/2} }{ \Big( \sum_{ i =1}^n |Z_i|^p + W \Big)^{1/p} \Big( \sum_{ i =1}^n g_i^2 \Big)^{1/2}  }\,,
\end{align*}
where $Z_i$ are $p$-Gaussian random variables, $W$ has distribution $\mathbb{W}$, $g_i$ are standard Gaussian random variables, and $I^{(\lambda_n)}_{i}$ are Bernoulli random variables with parameter $\lambda_n$. 

Recentering the random variables so the sums have zero expectation, we may write
\begin{align*}
|| P_{E_n} X_n ||_2 & \stackrel{\dint}{=} \frac{ \left( n M_p(2) +  \sum_{ i =1}^n (Z_i^2 - M_p(2) ) \right)^{1/2} \left( n  \lambda_n +  \sum_{ i =1}^{n} (  I^{(\lambda_n)}_{i}\, g_i^2 - \lambda_n)  \right)^{1/2} }{ \left(  n  + \sum_{ i =1}^n ( |Z_i|^p -1) + W \right)^{1/p} \left( n + \sum_{ i =1}^n (g_i^2 - 1) \right)^{1/2}  }\\
&= n^{1/2 - 1/p} \sqrt{ \lambda_n M_p(2) }  F \left( \frac{a_n}{ n M_p(2) }, \frac{b'_n}{ n \lambda_n}, \frac{c_n}{ n} , \frac{W}{n} , \frac{ d_n}{ n}  \right),
\end{align*}
where $F$ and $a_n,c_n,d_n$ are as introduced in Section \ref{sec:RP 1 setup}, and 
\begin{align*}
b'_n := \sum_{ i =1}^n \left( I^{(\lambda_n)}_{i} g_i^2  - \lambda_n \right).
\end{align*}
With $\Psi$ as in Section \ref{sec:RP 1 setup}, it follows that the random variable $Y_n := \frac{n^{1/p}}{ \sqrt{M_p(2) }} || P_{E_n} X_n||_2 - \sqrt{\lambda_n n}$ may be written
\begin{align*}
Y_n := \xi_n - \frac{ \sqrt{\lambda_n } W}{ p \sqrt{n} } + \sqrt{ \lambda_n n } \Psi \left( \frac{a_n}{ n M_p(2) }, \frac{b'_n}{ n \lambda_n}, \frac{c_n}{ n} , \frac{w}{n} , \frac{ d_n}{ n}  \right),
\end{align*}
where $\xi_n = \sqrt{\lambda_n n} \left( \frac{a_n}{ 2 n M_p(2)}  +  \frac{b'_n}{ 2 n \lambda_n} - \frac{c_n}{  p n}  -   \frac{ d_n}{2  n}  \right)$. The goal of this section is to obtain an upper bound on the Kolmogorov-Smirnov distance between $Y_n$ and a Gaussian random variable $G$ with variance $w(\lambda)$, where
\begin{align} \label{w def}
w(s) := s \sigma^2 (p,2) + \frac{3}{4} (1-s), \qquad s \in [0,1].
\end{align} 
Again using Lemma \ref{separating} to obtain the first inequality below, and the triangle inequality \eqref{berry triangle}  for Kolmogorov-Smirnov distances to obtain the second, for any $\epsilon > 0$, we have
\begin{align}
\dint_{\mathsf{KS}}(Y_n, G) & \leq \dint_{\mathsf{KS}}(\xi_n, G) \nonumber  
 + \P \left[  \frac{ \sqrt{\lambda_n } W}{ p \sqrt{n} }  > \frac{ \epsilon}{ 2 } \right] +  \P \left[ \left| \Psi \left( \frac{a_n}{n M_p(2) },  \frac{b'_n}{ n \lambda_n} , \frac{c_n}{n }, \frac{ W}{ n} , \frac{ d_n}{ n} \right) \right| > \frac{ \epsilon}{ 2 \sqrt{n \lambda_n} }  \right] \cr
 &\qquad+ \frac{ \epsilon}{ \sqrt{2 \pi w(\lambda) } } \nonumber  \\
& \leq \dint_{\mathsf{KS}}(\xi_n, G_n)  + \dint_{\mathsf{KS}}(G_n, G) 
+  \P \left[  \frac{ \sqrt{\lambda_n } W}{ p \sqrt{n} }  > \frac{ \epsilon}{ 2 } \right] \cr 
& \qquad +  \P \left[ \left| \Psi \left( \frac{a_n}{n M_p(2) },  \frac{b'_n}{ n \lambda_n} , \frac{c_n}{n }, \frac{ W}{ n} , \frac{ d_n}{ n} \right) \right| > \frac{ \epsilon}{ 2 \sqrt{\lambda_n n }}  \right] + \frac{ \epsilon}{ \sqrt{2 \pi  J_p } } \label{60}\,,
\end{align}
where $G_n$ is a centered Gaussian random variable with the same variance as $\xi_n$, and 
\begin{align*}
 J_p := \inf_{s \in [0,1]} w(s) = \min \{ \sigma^2 (p,2), 3/4\}
\end{align*}
is positive. 

The remainder of the proof is thus split into four parts. First we use the Berry-Esseen theorem to obtain an upper bound for $\dint_{\mathsf{KS}}(\xi_n,G_n)$. Then we use Proposition \ref{berry gauss lip} to give an upper bound on $\dint_{\mathsf{KS}}(G_n,G)$. Then finally we take a choice of $\epsilon = \epsilon_n$ to give an upper bound for the term involving $\Psi$. In the final part we collect our bounds together to prove the result.   

\subsection{A bound for $\dint_{\mathsf{KS}}(\xi_n,G_n)$}
A calculation establishes that $\xi_n$ may be written as a sum $\xi_n = \frac{1}{ \sqrt{n}} \sum_{ i = 1}^n  \varphi_i^{(\lambda_n)}$ of independent and identically distributed centered random variables distributed like $\varphi^{(\lambda_n)}$, where for each $s \in (0,1]$, $\varphi^{(s)}$ is the centered random variable
\begin{align*}
\varphi^{(s)} :=  \sqrt{s} \left( \frac{ Z^2 - M_p(2) }{ 2 M_p(2) } + \frac{ I^{(s)} g^2 - s  }{ 2 s } - \frac{ |Z|^p - 1}{ p} - \frac{ g^2 - 1}{ 2}  \right)
\end{align*}
where $g$ is a standard Gaussian, $Z$ is $p$-Gaussian, and $I^{(s)}$ is Bernoulli distributed with $\P[I^{(s)} = 1 ] = s = 1 - \P [I^{(s)}= 0 ]$. 

We now take a moment to investigate the behavior of the moments of $\varphi^{(s)}$ as $s$ varies in $(0,1]$, with particular care for the case $s \to 0$. Considering first the second moment, a calculation shows that 
$\E[ |\varphi^{(s)} |^2 ] = w(s) $,
where $w(s)$ is given in \eqref{w def}.  As for the third moment, we may write 
\[
\varphi^{(s)} = \sqrt{s} \mathcal{W}_p + \frac{1}{2\sqrt{s}}I^{(s)} g^2,
\] 
where $\mathcal{W}_p$ is a random variable independent of $I^{(s)}$ and whose distribution is independent of $s$. Using the definition of $I^{(s)}$ and the independence to obtain the second equality below, we have
\begin{align*}
\E \big[ |\varphi^{(s)}|^3 \big] &= \frac{1}{s^{3/2}}\, \E \left[ \left|\frac{I^{(s)} g^2}{2} + s \mathcal{W}_p \right|^3 \right]\\
&= \frac{1}{\sqrt{s}}\, \E \left[ \left| \frac{g^2}{2} + s \mathcal{W}_p \right|^3 \right] + \frac{1}{s^{3/2}}\, (1-s)\, \E\big[ | s \mathcal{W}_p|^3 \big] \\
&= \frac{1}{\sqrt{s}}\, f_p(s)
\end{align*} 
for a bounded function $f_p:[0,1] \to [0,\infty)$. Since $w(s) := \E[ | \varphi^{(s)}|^2]$ is bounded below on $[0,1]$ by $J_p > 0$, and $\sqrt{s} \E [ |\varphi^{(s)}|^3 ] = f_p(s)$ is bounded above on $[0,1]$, there is a constant $C_p\in(0,\infty)$ such that
\begin{align*}
C_p  \geq \sup_{ s \in [0,1]}  \sqrt{s}\, \frac{\E [ |\varphi^{(s)}|^3 ]  }{\E[ |\varphi^{(s)} |^2 ]   } .
\end{align*}
We are now ready to use the Berry-Esseen theorem. Recall that $\xi_n := \frac{1}{ \sqrt{n}} \sum_{ i = 1}^n \varphi_i^{(\lambda_n)}$ is a rescaled sum of independent and identically distributed random variables with law $\varphi^{(\lambda_n)}$. If $G_n$ is a Gaussian random variable with variance $w(\lambda_n)$, then using the the Berry-Esseen theorem to obtain the first inequality below, and the definition of $C_p$ to obtain the second, we obtain
\begin{align} \label{6a}
\dint_{\mathsf{KS}} (\xi_n , G_n ) \leq K_{\mathsf{BE}} \frac{\E [ |\varphi^{(\lambda_n)}|^3 ]  }{\E[ |\varphi^{(\lambda_n)} |^2 ]   } \frac{1}{ \sqrt{n}} \leq K_{\mathsf{BE}} \frac{ C_p }{ \sqrt{ \lambda_n n} } \leq \frac{C_p'}{ \sqrt{\lambda_n n }} ,
\end{align}
where $C_p' := K_{\mathsf{BE}}\, C_p\in(0,\infty)$ depends only on $p$.

\subsection{A bound for $\dint_{\mathsf{KS}}{(G_n,G)}$}
The random variables $G_n$ and $G$ are centered Gaussians with variances $w(\lambda_n)$ and $w(\lambda)$ respectively. Since $w$ is Lipschitz on $[0,1]$, and bounded below on $[0,1]$ by a positive constant, it follows from Proposition \ref{berry gauss lip} that there is a constant $C_p\in(0,\infty)$ depending on $p$ but not on $n,\lambda_n,\lambda$ such that 
\begin{align} \label{6b}
\dint_{\mathsf{KS}}( G_n, G) \leq \frac{3}{8} \frac{ | w(\lambda_n) - w(\lambda) | }{ w(\lambda) } \leq C_p | \lambda_n - \lambda|.
\end{align}

\subsection{Bounding the $\epsilon$-dependent terms}
Let $(\epsilon_n)_{n \geq 1}$ be a sequence of positive reals decreasing down to $0$. By
Lemma \ref{lem:Psi bound} and the triangle inequality, with $v_n :=  \left( \frac{a_n}{n M_p(2) },  \frac{b'_n}{ n \lambda_n} , \frac{c_n}{n }, \frac{ W}{ n} , \frac{ d_n}{ n} \right) $ we have
\begin{align}
& \P \left[ \left| \Psi \left( v_n\right) \right| > \frac{ \epsilon_n }{ 2 \sqrt{\lambda_n n } }  \right]  \leq  \P \left[ \left| \left| v_n \right|  \right|_2 \geq  \sqrt{ \frac{ \epsilon_n  }{ 2M \sqrt{\lambda_n n } } }  \right] + 
\P \big[ \left| \left|   v_n \right|  \right|_2 \geq  \delta \big] \nonumber \\
&\leq Q_n \left( \frac{a_n}{ n M_p(2) } \right) +  Q_n \left( \frac{b'_n}{n \lambda_n} \right) +   Q_n \left(  \frac{c_n}{n } \right) +  Q_n \left( \frac{ W}{ n } \right)  +  Q_n \left(  \frac{ d_n}{ n} \right),  \label{apple}
\end{align} 
where for a random variable $X$,
\begin{align*}
Q_n(X) := \P \left[ |X| > \frac{\delta}{\sqrt{5}} \right] + \P \left[ | X | >   \sqrt{ \frac{ \epsilon_n }{ 10 M \sqrt{\lambda_n n } } }   \right]. 
\end{align*}
Considering first the terms involving $a_n, b_n$, and $d_n$, we may use the bounds developed in Section 3.4. In particular there are sufficiently large constants $C_p\in(0,\infty)$ and $\Theta_p\in(0,\infty)$, such that by setting $\epsilon_n = \Theta_p \frac{ \log(k_n)}{ \sqrt{k_n}}$, we ensure that 
\begin{align} \label{outcome 3}
Q_n\left( \frac{a_n}{ n M_p(2) } \right) + Q_n\left( \frac{c_n}{n} \right) + Q_n \left( \frac{d_n}{ n } \right) \leq \frac{C_p}{\sqrt{\lambda_n n} }
\end{align}
for every $n\in\N$. 

Now we look to obtain a bound for $Q_n \left( \frac{b'_n}{n \lambda_n} \right)$, which may be written
\begin{align*}
Q_n \left( \frac{b'_n}{n \lambda_n} \right) = \P \left( \left| \frac{b_n'}{ \sqrt{\lambda_n n } } \right| > c_p \sqrt{ n \lambda_n } \right) + \P \left( \left| \frac{b_n'}{ \sqrt{\lambda_n n } } \right| > c'_p \sqrt{ \epsilon_n \sqrt{n \lambda_n } } \right),
\end{align*}
where $c_p := \delta/ \sqrt{5}$ and $c_p' := \frac{1}{ \sqrt{10 M }}$. Recalling that $b'_n := \sum_{ i =1}^n \left( I^{(\lambda_n)}_{i} g_i^2  - \lambda_n \right)$ is a sum of independent and identically distributed random variables, we see that we are in the setting of Lemma \ref{gaussish bound}, with tails
\begin{align*}
\alpha_n :=  c_p  \sqrt{ n \lambda_n } \qquad \text{and} \qquad \alpha_n' :=  c'_p \sqrt{ \epsilon_n \sqrt{n \lambda_n } } .
\end{align*}
In order to use Lemma \ref{gaussish bound}, first we need to take a moment to study the moments of the random variables in $b_n'$. Let $I^{(s)}$ have the Bernoulli distribution with parameter $s$, $g$ be a standard Gaussian random variable, and define the random variable $\phi^{(s)} := \frac{1}{ \sqrt{s}} \left( I^{(s)} g^2 - s \right)$.
Then
\begin{align} \label{S_n def}
S_n := \frac{ b_n'}{ \sqrt{ \lambda_n n} } = \frac{1}{ \sqrt{n}} \sum_{ i = 1}^n \phi_i^{(\lambda_n)},
\end{align}
where $\phi_i^{(\lambda_n)}$ are independent and identically distributed like $\phi^{(s)}$. Using the fact that $\E[g^2] = 1,   \E[g^4 ] = 3, \E[g^6] = 15$, it is straightforward to check that $\phi^{(s)}$ is centered and has absolute second and third moments given by
\begin{align*}
\E\big[ | \phi_s |^2 \big] = (3-s) \qquad\text{and}\qquad  \E\big[ |\phi_s|^3 \big] =  s^{ - 3/2} \left( 15s - 9s^2 + 4s^3 - 2s^4 \right),
\end{align*}
the latter of which is bounded above by $19/\sqrt{s}$ for $s \in [0,1]$.
Using the fact that $\E[ |\phi_s|^2 ]\geq 2$ on $[0,1]$, we have the bound
\begin{align*}
\sup_{ s \in [0,1] } \sqrt{s}\, \frac{ \E\big[ |\phi_s|^3 \big]  }{ \E\big[ | \phi_s |^2 \big] ^{3/2} } \leq \frac{19}{ 2^{3/2}}.
\end{align*}
It follows that with the sum $S_n$  in \eqref{S_n def}, we are in the set up of Lemma \ref{gaussish bound} with
\begin{align*}
\gamma_n =  \frac{19}{2^{3/2}}  \frac{1}{\sqrt{\lambda_n}},
\end{align*} 
and $\sigma_{\mathsf{max}} = 3 \geq  \E[ |\phi_{\lambda_n}|^2 ] $ for all $\lambda_n$ and $n\in\N$. 

Now let $\beta_n$ and $\epsilon_n$ be sequences such that $n/\gamma_n^2 \geq \beta_n$ for every $n\in\N$, and additionally
\begin{align} \label{ineq orange}
c_p  \sqrt{ n \lambda_n }  \geq \sigma_{\mathsf{max}}  \sqrt{ \log ( \beta_n)}_+ \qquad \text{and} \qquad c'_p \sqrt{ \epsilon_n \sqrt{n \lambda_n } }  \geq  \sigma_{\mathsf{max}}  \sqrt{ \log ( \beta_n)}_+
\end{align}
for every $n\in\N$. Then by Lemma \ref{gaussish bound} we have
\begin{align*}
Q_n \left( \frac{b'_n}{ n \lambda_n} \right) \leq \frac{ 4 ( K_{\mathsf{BE}} + 1) }{ \sqrt{\beta_n}}.
\end{align*} 
In particular, consider setting $\beta_n = \theta'_p \lambda_n n $ and $\epsilon_n = \Theta'_p \frac{ \log(\lambda_n n )}{ \sqrt{n \lambda_n }}$. Then if $\theta'_p$ is sufficiently small and $\Theta'_p$ is sufficiently large such that $n/\gamma_n^2 \geq \beta_n$ and the inequalities \eqref{ineq orange} hold for every $n\in\N$, then for a sufficiently large constant $C'_p\in(0,\infty)$ we have
\begin{align} \label{outcome 4}
Q_n \left( \frac{b'_n}{n \lambda_n} \right)  \leq \frac{C_p'}{\sqrt{n \lambda_n \theta}}.
\end{align}
In particular, setting $\Theta^*_p := \max \{ \Theta_p , \Theta_p' \}$, both \eqref{outcome 3} and \eqref{outcome 4} hold, and there is a constant $C_p^*\in(0,\infty)$ such that
\begin{align} \label{orange}
Q_n\left( \frac{a_n}{ n M_p(2) } \right) + Q_n\left( \frac{c_n}{n} \right) + Q_n \left( \frac{d_n}{ n } \right) + Q_n \left( \frac{b'_n}{n \lambda_n} \right) \leq \frac{C^*_p}{\sqrt{\lambda_n n}}.
\end{align}
With $ \epsilon_n :=\Theta_p^* \frac{ \log( \lambda_n n) }{ \sqrt{ n \lambda_n } }$, by plugging \eqref{orange} into \eqref{apple} we obtain
\begin{align} \label{6c}
 &\P \left[ \left| \Psi  \left( \frac{a_n}{n M_p(2) },  \frac{b'_n}{ n \lambda_n} , \frac{c_n}{n }, \frac{ W}{ n} , \frac{ d_n}{ n} \right) \right| > \frac{ \epsilon_n }{ 2 \sqrt{\lambda_n n } }  \right]\\
&  \leq \frac{C_p^*}{\sqrt{ n \lambda_n }} + \P \left[ W > \frac{n  \delta}{\sqrt{5}}  \right] + \P \left[ W >n   \sqrt{ \frac{ \epsilon_n }{ 10M \sqrt{\lambda_n n } } }   \right]\,.
\end{align}

\subsection{Wrapping things together}
Considering the original bound \eqref{60}, and the bounds for the respective parts \eqref{6a}, \eqref{6b}, and \eqref{6c}. With the choice $\epsilon_n = \Theta^*_p \frac{ \log(n) }{ \sqrt{n}}$, there is some constant $c_p\in(0,\infty)$ such that the terms involving tail  probabilities for $W$ in \eqref{60} and \eqref{6c} can be bounded by writing
\begin{align*}
\P \left[  \frac{ \sqrt{\lambda_n } W}{ p \sqrt{n} }  > \frac{ \epsilon_n}{ 2 } \right] +  \P \left[ W > \frac{n  \delta}{\sqrt{5}}  \right] + \P \left[ W >n   \sqrt{ \frac{ \epsilon_n }{ 10M \sqrt{\lambda_n n } } }   \right] \leq 3\, \P \left[ W > c_p \frac{1}{ \lambda_n } \log(n) \right],
\end{align*}
where we used the fact that $W$ is non-negative. 

In particular, combining \eqref{6a}, \eqref{6b}, and \eqref{6c} in \eqref{60}, it follows that there exist constants $c_p,C_p\in(0,\infty)$ depending only on $p$ such that 
\begin{align*}
\dint_{\mathsf{KS}}( Y_n  , G) \leq C_p  \max \left\{ \frac{ \log(n) }{ \sqrt{ \lambda_n n } } , | \lambda_n - \lambda| \right\} + 3 \,\P \left[W >  c_p \frac{1}{ \lambda_n } \log(n) \right] ,
\end{align*}
completing the proof of Theorem \ref{thm:RP 2}.

\section{Proof of Theorem \ref{thm:q norm} }\label{sec:q norm proof}
In this section we prove Theorem \ref{thm:q norm}, which establishes Berry-Esseen type bounds for the $q$-norm of a vector-valued random variable with law $\mathbb{P}_{n,p,\mathbb{W}}$.
\subsection{Setting up the proof}
 Let $X_n$ have law $\mathbb{P}_{n,p,\mathbb{W}}$. Then $X_n$ has a representation
\begin{align*}
X_n \stackrel{\dint}{=} \frac{ Z }{ \left(  \sum_{i= 1}^n |Z_i|^p + W \right)^{1/p} } ,
\end{align*}
where $Z_1,\ldots,Z_n$ are independent and identically distributed $p$-Gaussian random variables, and $W$ is independent of the $Z_i$ with distribution $\mathbb{W}$. In particular, the $q$-norm of $X_n$ may be written
\begin{align}
||X_n||_q &\stackrel{\dint}{=}  \frac{ \left( \sum_{i=1}^n |Z_i|^q \right)^{1/q}  }{ \left(  \sum_{i= 1}^n |Z_i|^p + W \right)^{1/p} }  \nonumber  \\
 & = \frac{ \left(  nM_p(q) + \sum_{i=1}^n (|Z_i|^q-M_p(q)) \right)^{1/q}  }{ \left(  n +  \sum_{i= 1}^n (|Z_i|^p - 1)  + W \right)^{1/p} }   \nonumber  \\
&= \frac{ (nM_p(q))^{1/q} }{ n^{1/p} } F \left( \frac{ \xi^1_n}{ \sqrt{n}}, \frac{ \xi^2_n }{\sqrt{n}}, \frac{ W}{ n } \right), \label{X rep}
\end{align}
where $F(a,b,c) = \frac{ (1+a)^{1/q} }{ (1 + b + c)^{ 1/p} }$, and $\xi^1_n$ and $\xi^2_n$ are given by 
\begin{align*}
\xi_n^1 = \frac{1}{ \sqrt{n}} \sum_{i = 1 }^n \frac{ |Z_i|^q - M_p(q)}{ M_p(q)}\qquad\text{and}\qquad \xi_n^2 = \frac{1}{ \sqrt{n}} \sum_{i = 1 }^n \frac{ |Z_i|^p - 1}{ 1}\,.
\end{align*}
Since $F$ is twice differentiable in a neighborhood about the origin in $\mathbb{R}^3$, we may write
\begin{align} \label{F rep}
F(a,b,c) = 1 + \frac{1}{p} a - \frac{1}{q} (b + c) + \Psi(a,b,c),
\end{align}
where $\Psi:\mathbb{R}^3 \to \mathbb{R}$ is such that there exists $M, \delta \in(0,\infty)$ with
\begin{align} \label{Psi bound 1}
|| \Psi(x)|| \leq M ||x||_2^2 \quad \text{whenever} \quad ||x||_2 < \delta.
\end{align}
Combining \eqref{X rep} and \eqref{F rep}, we have
\begin{align*}
Y_n &:= \sqrt{n} \left( n^{1/q - 1/p} M_p(q)^{1/q} ||X_n||_q - 1 \right)\\
&= \xi_n - \frac{ W }{ p  \sqrt{n} } + \sqrt{n} \Psi  \left( \frac{ \xi^1_n}{ \sqrt{n}}, \frac{ \xi^2_n }{\sqrt{n}}, \frac{ W}{ n } \right),
\end{align*}
where $\xi_n = \frac{1}{p} \xi_n^1 - \frac{1}{q} \xi_n^2$. Let $G$ be a Gaussian random variable with variance $\sigma^2(p,q)$. In light of Lemma \ref{separating}, for any $\epsilon > 0$ we have the following bound
\begin{align} 
&\dint_{\mathsf{KS}} \left( Y_n , G \right) \nonumber \\
& \leq \dint_{\mathsf{KS}}( \xi_n , G) + \P \left[  \frac{ W}{ n }  > \epsilon/2 \right] + \P \left[ \left| \sqrt{n} \Psi  \left( \frac{ \xi^1_n}{ \sqrt{n}}, \frac{ \xi^2_n }{\sqrt{n}}, \frac{ W}{ n } \right) \right| > \epsilon /2  \right] + \frac{ \epsilon}{ \sqrt{ 2 \pi \sigma^2(p,q) } }.  \label{full eq}  
\end{align}
To give a brief overview of the remainder of the proof, first we use the Berry-Esseen theorem to show that the random variable $\xi_n$ is close to a Gaussian random variable with variance $\sigma^2(p,q)$ in terms of Kolmogorov-Smirnov distances, obtaining an upper bound for  $\dint_{\mathsf{KS}}( \xi_n , G) $. Then we obtain an upper bound for remaining terms. 

\subsection{A bound for $\dint_{\mathsf{KS}}(\xi_n , G)$ }
First of all, we note the random variable $\xi_n$ may be written $\xi_n = \frac{1}{\sqrt{n}} \sum_{ i=1}^n \mathcal{Z}_i$, where
\begin{align*}
\mathcal{Z}_i :=  \frac{ |Z_i|^q - M_p(q)}{ q M_p(q)} - \frac{ |Z_i|^p - 1}{ p },\qquad i\in\{1,\dots,n\}.
\end{align*} 
Each $\mathcal{Z}_i$ has variance $\sigma^2(p,q)$ as in \eqref{variance}. In particular, $\xi_n$ also has variance $\sigma^2(p,q)$. Let $\rho(p,q)$ denote the third moment of $|\mathcal{Z}_i|$. Then by the Berry-Esseen theorem (see Theorem \ref{thm:berry esseen}), if $G$ is a centered Gaussian random variable with variance $\sigma^2(p,q)$, then we have
\begin{align} \label{gauss eq} 
 \dint_{\mathsf{KS}}( \xi_n , G)  \leq K_{\mathsf{BE}} \frac{ \rho(p,q) }{ \sigma^3(p,q) \sqrt{n}} =: \frac{ C_{p,q} }{ \sqrt{n}}.
\end{align}

\subsection{A bound for the $\Psi$ term}
Now consider obtaining an upper bound for the term involving $\Psi$ in \eqref{full eq}. 
With $M$ and $\delta$ as in \eqref{Psi bound 1}, using Lemma \ref{lem:Psi bound} and the triangle inequality, we have
\begin{align} \label{mango}
\P \left[ \left| \sqrt{n} \Psi  \left( \frac{ \xi^1_n}{ \sqrt{n}}, \frac{ \xi^2_n }{\sqrt{n}}, \frac{ W}{ n } \right) \right| > \frac{\epsilon_n}{2}  \right]  \leq Q_n \left(  \frac{ \xi^1_n}{ \sqrt{n}} \right) + Q_n \left(  \frac{ \xi^2_n}{ \sqrt{n}} \right) + Q_n \left(   \frac{W}{n}   \right),
\end{align}
where for real-valued random variables $X$,
\begin{align*}
Q_n \left( X \right) := \P\left[ |X| > \sqrt{ \frac{ \epsilon_n}{ 6M \sqrt n }} \right] + \P \left[ | X| > \frac{\delta}{\sqrt{3}} \right].
\end{align*}
Let us start with an upper bound for 
\begin{align*}
Q_n \left(  \frac{ \xi^1_n}{ \sqrt{n}} \right)  =  \P\left[ |\xi^1_n| > \sqrt{ \frac{ \epsilon_n \sqrt{n}}{ 6M  }} \right] + \P\left[ |\xi^1_n| > \frac{\delta \sqrt{n }}{\sqrt{3}} \right] \,.
\end{align*}
We may use Lemma \ref{gaussish bound} to bound the tail probabilities of $\xi^1_n$. To set this up, let $m=n,  \beta_n = \theta n$ for some $\theta \in(0,1]$, 
\begin{align*}
\gamma_n  := \frac{ \E \left[ \left|  \frac{ |Z_1|^q - M_p(q)}{ M_p(q)} \right|^3 \right] }{ \E  \left[ \left|  \frac{ |Z_1|^q - M_p(q)}{ M_p(q)} \right|^2 \right]^{3/2} } \qquad \text{for every $n\in\N$},
\end{align*}
and let $\sigma_{\mathsf{max}}$ be the variance of $ \frac{ |Z_1|^q - M_p(q)}{ M_p(q)} $. In particular, whenever $\epsilon_n$ and $\beta_n \leq n/ \gamma_n^2$ are sequences such that
\begin{align} \label{ineq orange}
\sqrt{  \frac{ \epsilon_n \sqrt{n}}{ 6M }} \geq \sigma_{\mathsf{max}} \sqrt{ \log( \beta_n) }_+ \qquad \text{and} \qquad \frac{\delta \sqrt{n}}{\sqrt{3}} \geq \sigma_{\mathsf{max}} \sqrt{ \log( \beta_n) }_+
\end{align}
for every $n\in\N$, then
\begin{align} \label{consequence orange}
Q_n \left(  \frac{ \xi^1_n}{ \sqrt{n}} \right) \leq \frac{ 4 \left( K_{ \mathsf{BE}} + 1 \right)}{ \sqrt{\beta_n}}. 
\end{align}
In particular, consider setting $\beta_n = \theta_{p,q} n$, and $\epsilon_n := \Theta_{p,q} \frac{ \log(n)}{ \sqrt{n}}$, where $\theta_{p,q}$ is small enough and $\Theta_{p,q}$ is large enough so that both the inequalities in \eqref{ineq orange} hold for every $n\in\N$. Then there is a constant $C_{p,q}\in(0,\infty)$ such that with  $\epsilon_n := \Theta_{p,q} \frac{ \log(n)}{ \sqrt{n}}$, we have
\begin{align*}
Q_n \left(  \frac{ \xi^1_n}{ \sqrt{n}} \right) \leq \frac{C_{p,q}}{\sqrt{n}}
\end{align*}
for every $n\in\N$. Making an identical argument with the term $Q_n \left( \xi_n^2 / \sqrt{n} \right)$, it follows that there are sufficiently large constants $\Theta_{p,q}'$ and $C_{p,q}'$ such that taking $\epsilon_n := \Theta_{p,q}' \frac{ \log(n)}{ \sqrt{n}}$, we have
\begin{align*}
 Q_n \left(  \frac{ \xi^1_n}{ \sqrt{n}} \right) + Q_n \left(  \frac{ \xi^2_n}{ \sqrt{n}} \right)  \leq \frac{C_{p,q}'}{\sqrt{n}}
\end{align*}
for every $n\in\N$. Hence, by \eqref{mango} and the definition of $Q_n(W)$, with $\epsilon_n := \Theta_{p,q}' \frac{ \log(n)}{ \sqrt{n}}$ we have
\begin{align}
&\P \left[\left| \sqrt{n} \Psi  \left( \frac{ \xi^1_n}{ \sqrt{n}}, \frac{ \xi^2_n }{\sqrt{n}}, \frac{ W}{ n } \right) \right| > \frac{\epsilon_n}{2}  \right] \nonumber  \\
&\leq \frac{C_{p,q}}{\sqrt{n}} + \P \left[ W > \frac{\delta n}{\sqrt{3}} \right] + \P \left[W >  n \sqrt{ \frac{ \epsilon_n}{ 6M \sqrt n }}\right].  \label{gauss eq 2}
\end{align}

\subsection{Wrapping things together}
Recall our original bound \eqref{full eq} for $\dint_{\mathsf{KS}}(Y_n, G)$. Plugging the intermediate bounds \eqref{gauss eq}  and \eqref{gauss eq 2} into \eqref{full eq}, we recall that there are constants $\Theta_{p,q}', C_{p,q}, C_{p,q}' \in(0,\infty)$ such that with $\epsilon_n := \Theta_{p,q}' \frac{ \log(n)}{ \sqrt{n}}$ we have
\begin{align} \label{cucumber}
&\dint_{\mathsf{KS}} \left( Y_n , G \right) \\
& \leq \frac{ C_{p,q} }{ \sqrt{n}} + \frac{ C_{p,q}'}{  \sqrt{n} }  + \P \left[ W > \frac{\delta n}{\sqrt{3}} \right] + \P \left[ W >  n \sqrt{ \frac{ \epsilon_n}{ 6M \sqrt n }}\right]+ \P \left[  \frac{ W}{ n }  > \frac{\epsilon_n}{2}\right]  + \frac{ \epsilon_n}{ \sqrt{ 2 \pi \sigma^2(p,q) } }.  \nonumber
\end{align}
For a sufficiently small constant $c_{p,q}\in(0,\infty)$, the terms in \eqref{apple} involving tail probabilities for $W$ may be collected by the simple bound
\begin{align} \label{pineapple}
\P \left[ W > \frac{\delta n}{\sqrt{3}} \right] + \P \left[ W >  n \sqrt{ \frac{ \epsilon_n}{ 6M \sqrt n }}\right] + \P \left[  \frac{ W}{ n }  > \frac{\epsilon_n}{2} \right] \leq \P \left[ W > c_{p,q} \sqrt{ n  \log(n)} \right],
\end{align} 
and it follows from plugging \eqref{pineapple} and the definition $\epsilon_n := \Theta_{p,q}' \frac{ \log(n)}{ \sqrt{n}}$ into \eqref{apple} that with $c_{p,q}$ as above, for a sufficiently large $C''_{p,q} \in(0,\infty)$ we have
\begin{align*}
\dint_{KS}(Y_n,G) \leq  C_{p,q}''  \frac{ \log(n)}{ \sqrt{n}} + \P \left[ W > c_{p,q} \sqrt{ n  \log(n)} \right],
\end{align*}
which completes the proof of Theorem \ref{thm:q norm}.

\subsection*{Acknowledgment} 
SJ and JP have been supported by the Austrian Science Fund (FWF) Project P32405 ``Asymptotic Geometric Analysis and Applications'' of which JP is principal investigator. JP has also been supported by a \textit{Visiting International Professor Fellowship} from the Ruhr University Bochum and its Research School PLUS.

\bibliographystyle{plain}
\bibliography{berry_esseen}

\end{document}